\documentclass[10pt]{amsart}
\usepackage{amssymb,amsmath}
\usepackage{hyperref}
\usepackage{bm}
\usepackage[]{graphicx}
\usepackage{amssymb, epsfig}
\usepackage{float}
\newtheorem{thm}{Theorem}[section]
\newtheorem{remm}{Remark}[section]
\newtheorem{lemma}{Lemma}[section]

\newtheorem{remark}[remm]{Remark}

\newtheorem{proposition}[thm]{Proposition}

\numberwithin{equation}{section}
\newcommand{\cset}{{\mathbb C}}

\newcommand{\nset}{{\mathbb N}}

\newcommand{\cunit}{\text{\rm i}}
\newcommand{\ssy}{\scriptscriptstyle}

\newtheorem{rem}[remm]{Remark}
\begin{document}
\title{A Finite Difference method for the\\
Wide-Angle `Parabolic' equation in a waveguide\\
with downsloping bottom}
\author{
D.C.~Antonopoulou$^{\dag\P}$,
V.A.~Dougalis$^{\ddag\P}$ and
G.E.~Zouraris$^{\S\P}$}
\thanks
{$^{\ddag}$ Department of Mathematics, University of Athens,
Panepistimiopolis, GR--157 84 Zographou, Greece.}
\thanks
{$^{\dag}$ Department of Applied Mathematics, University of Crete,
GR--714 09 Heraklion, Greece.}
\thanks
{$^{\S}$ Department of Mathematics, University of Crete, GR--714
09 Heraklion, Greece.}
\thanks
{$^{\P}$ Institute of Applied and Computational Mathematics,
FO.R.T.H., GR--711 10 Heraklion, Greece.}
\subjclass{65M06, 65M12, 65M15, 76Q05}
\keywords{wide-angle Parabolic Equation, underwater sound
propagation, variable domains, downsloping bottom,
initial-boundary-value problems, finite difference methods,
error estimates.}
\begin{abstract}
We consider the third-order wide-angle `parabolic' equation of
underwater acoustics in a cylindrically symmetric fluid medium
over a bottom of range-dependent bathymetry. It is known that the
initial-boundary-value problem for this equation may not be well
posed in the case of (smooth) bottom profiles of arbitrary shape
if it is just posed e.g. with a homogeneous Dirichlet bottom
boundary condition. In this paper we concentrate on downsloping
bottom profiles and propose an additional boundary condition that
yields a well posed problem, in fact making it $L^2$-conservative
in the case of appropriate real parameters. We solve the problem
numerically by a Crank-Nicolson-type finite difference scheme,
which is proved to be unconditionally stable and second-order
accurate, and simulates accurately realistic underwater acoustic
problems.
\end{abstract}
\maketitle
%
%
%
%
%
%
\section{Introduction}
We consider the third-order wide-angle `parabolic' equation of
underwater acoustics in a cylindrically symmetric fluid medium,
\cite{ref12,Claer,Lee,AkDZ1996}
\begin{equation}\tag{WA}
(1+q\,\beta)\,v_r+\alpha^2\,q\,v_{zzr}={\rm
i}\,\lambda\,(\alpha^2\,v_{zz}+\beta\, v),
\end{equation}
posed for $(r,z)\in\mathcal{D}:=\{(r,z)\in \mathbb{R}^2,\;0\leq
z\leq s(r),\;0\leq r\leq R\}$ for a given $R>0$ and a given bottom
profile $s=s(r)$.
\begin{figure}[htb]
\centering
\includegraphics[width=0.75\textwidth,height=0.20\textheight]{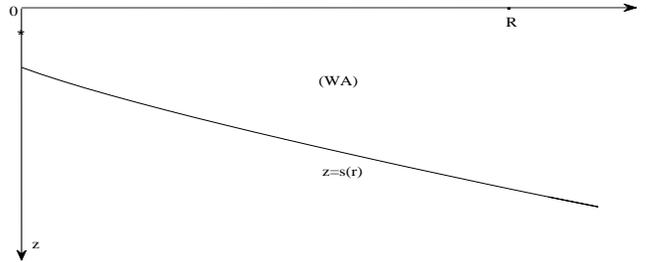}
\vskip-0.8truecm \caption{Domain of validity of (WA).}\label{f1}
\end{figure}
Here $v=v(r,z)$ is a complex-valued function of the range
$r$ and the depth $z$, representing the acoustic field generated
in the fluid medium (`water') $\mathcal{D}$ by a point
time-harmonic source of frequency ${\rm f}$ placed on the
$z$-axis, cf. Figure~\ref{f1}. In (WA) we have put
$\alpha=\frac{1}{k_0}$, where $k_0=\frac{2\pi {\rm f}}{c_0}$ is a
reference wave number and $c_0$ a reference sound speed, and
$\lambda=\frac{p-q}{\alpha}$, where $p$, $q$ are complex
constants. The function $\beta=\beta(r,z)$ is complex-valued and
represents $n^2-1$, where $n$ is the index of refraction of the
medium.
\par
In practice $\beta$ is real- or complex-valued with a small
nonnegative imaginary part modeling attenuation in the water
column. The coefficients $p$, $q$ are such that the rational
function $\frac{1+p\,x}{1+q\,x}$ is an approximation to
$\sqrt{1+x}$ near $x=0$. The p.d.e. (WA) is obtained formally as a
corresponding paraxial approximation of the outgoing
pseudodifferential factor of the Helmholtz equation written in
cylindrical coordinates in the presence of azimuthal symmetry,
\cite{ref12}, \cite{Lee}. The choice $p=\frac{1}{2}$, $q=0$ yields
the standard, narrow-angle Parabolic Equation, \cite{Tapp}, while
the $(1,1)$-Pad\'e approximant  of $\sqrt{1+x}$, given by
$p=\frac{3}{4}$, $q=\frac{1}{4}$, gives the Claerbout equation,
\cite{Claer}. In general, we shall take $p$ and $q$ complex; the
choice $p=q+\frac{1}{2}$, ${\rm Im}(q)<0$, \cite{Co}, has certain
theoretical and numerical advantages as will be seen in the
sequel. Let us also note that although in this paper we have in
mind the application of the (WA) in underwater acoustics
examples, our analytical and numerical methods can also be
applied to wide-angle seismic wave, \cite{Claer,Gr}, and
aeroacoustic \cite{CBJ} wave propagation.
\par
The p.d.e. (WA) is posed as an evolution equation with respect to
the time-like variable $r$ and the space-like variable $z$ in the
domain $\mathcal{D}$, that has a range-dependent bottom described
by $z=s(r)$, where $s$ is a smooth positive function on $[0,R]$.
We shall supplement (WA) by an initial condition modelling the
sound source at $r=0$, i.e. require that
\begin{equation}\label{d1.1}
v(0,z)=v_0(z),\quad 0\leq z\leq s(0),
\end{equation}
where $v_0$ is a given complex-valued function defined in
$[0,s(0)]$, and by the homogeneous Dirichlet boundary conditions
\begin{equation}\label{d1.2}
v(r,0)=0,\quad 0\leq r\leq R,
\end{equation}
\begin{equation}\label{d1.3}
v(r,s(r))=0,\;\;0\leq r\leq R,
\end{equation}
corresponding to a pressure-release surface and an acoustically
soft bottom, respectively.
\par
There is considerable theoretical and numerical evidence to the
effect that the initial-boundary-value problem (ibvp) consisting
of (WA), and \eqref{d1.1}-\eqref{d1.3} is well posed if the
bottom is horizontal or upsloping, i.e. when $\dot{s}(r)\leq 0$
in $[0,R]$, and that it may be ill posed if the bottom is
downsloping, i.e. if $\dot{s}(r)>0$ for $r\in[0,R]$,
\cite{Godin}, \cite{DSZP}, \cite{DSZ2009}, \cite{AnDSZ2008}. In
Refs. \cite{DSZP} and \cite{DSZ2009} an additional bottom
boundary condition was proposed, that together with (WA) and
\eqref{d1.1}-\eqref{d1.3} yields, under certain hypotheses, a
well posed problem for any smooth profile $s$.
\par
In \cite{AnDSZ2008} the authors of the paper at hand, in
collaboration with F.~Sturm, presented other types of additional
bottom boundary conditions that render the resulting ibvp well
posed and in addition, for real $\beta$ and $q$,
$L^2$-conservative, in the sense that
\begin{equation}\label{d1.4}
\int_0^{s(r)}|v(r,z)|^2\;dz=\int_0^{s(0)}|v_0(z)|^2\;dz
\end{equation}
holds for $0\leq r\leq R$. Specifically, it was observed that the
ibvp consisting of (WA) and \eqref{d1.1}-\eqref{d1.3} is
$L^2$-conservative, for real $\beta$ and $q$, if and only if the
following boundary value condition holds
\begin{equation}\label{d1.5}
\text{\rm Im}\left[\,{\mathcal F}(v_z;r,s(r))\,\overline{{\mathcal
F}(v;r,s(r))}\,\right]=0, \quad0\leq{r}\leq R,
\end{equation}
where for $(r,z)\in\mathcal{D}$
\begin{equation}\label{d1.6}
{\mathcal F}(v;r,z):=q\,v_r(r,z)-{\rm i}\,\lambda\, v(r,z),
\end{equation}
provided that $v$ satisfies (WA) and \eqref{d1.1}-\eqref{d1.3}.
Our main motivation for studying boundary conditions for which the
`energy' integral $\int_0^{s(r)}|v(r,z)|^2\;dz$ is conserved (for
real $\beta$ and $q$) in the case of the noncylindricral domain
$\mathcal{D}$ is the fact that this happens for solutions of (WA),
\eqref{d1.1}-\eqref{d1.3} in the case of a horizontal bottom and
also for the standard PE ($p=\frac{1}{2}$, $q=0$) posed on
$\mathcal{D}$ as an ibvp with \eqref{d1.1}-\eqref{d1.3}, for a
general profile $s$ and real $\beta$.
\par
Here, we consider the ibvp (WA), \eqref{d1.1}-\eqref{d1.3} for
a downsloping bottom, with the additional boundary condition
\begin{equation}\label{d1.7}
v_r(r,s(r))=0,\quad 0\leq r\leq R,
\end{equation}
which is equivalent to the condition $v_z(r,s(r))=0$, $r\in[0,R]$,
(in the case of a differentiable bottom with $\dot{s}(r)\neq 0$),
as it seen by differentiating both sides of \eqref{d1.3} with
respect to $r$. Obviously, \eqref{d1.7}, in the presence of
\eqref{d1.3}, is also equivalent to ${\mathcal F}(v;r,s(r))=0$,
i.e. satisfies \eqref{d1.5} and yields an $L^2$-conservative
problem for $\beta$, $q$ real. In section 2 of the present paper,
we prove that the resulting ibvp consisting of (WA),
\eqref{d1.1}-\eqref{d1.3}, \eqref{d1.7} is stable in $L^2$, $H^1$ and $H^2$;
also we show an $H^2$-stability result for the solution of the ibvp (WA),
\eqref{d1.1}-\eqref{d1.3} with an upsloping bottom.
\par
In order to develop a numerical method for ibvp (WA),
\eqref{d1.1}-\eqref{d1.3}, \eqref{d1.7} when the bottom is
downsloping and $q\not=0$, we transform it, using the
range-dependent change of the depth variable $y=\frac{z}{s(r)}$
that renders the bottom horizontal, into an equivalent problem on
the strip $0\leq y\leq 1$, $0\leq r\leq R$.
%
%
Letting $\Omega:=[0,R]\times[0,1]$, $I:=[0,1]$ and
%
$u(r,y):=v(r,y\,s(r))$ for $(r,y)\in\Omega$,
%
the ibvp consisting of (WA), \eqref{d1.1}-\eqref{d1.3},
\eqref{d1.7} takes the following form
\begin{equation}\label{NewProblem2}
\begin{split}
&\Lambda(r)\left(u_r-\cunit\,\tfrac{\lambda}{q}\,u-\tfrac{{\dot
s}(r)}{s(r)}\,y\,u_y\right)=\cunit\,\tfrac{\lambda\,s^2(r)}{\alpha^2\,q^2}\,u
\quad\forall\,(r,y)\in\Omega,\\
%
&u(r,0)=0\quad\forall\,r\in[0,R],\\
&u(r,1)=u_y(r,1)=0\quad\forall\,r\in[0,R],\\
&u(0,y)=u_0(y):=v_0(y\,s(0))\quad\forall\,y\in I,\\
\end{split}
\end{equation}
where $\gamma(r,y):=b(r,y s(r))$ for $(r,y)\in\Omega$ and, for
$r\in[0,R]$, $\Lambda(r):H^2(I)\cap H^1_0(I)\rightarrow L^2(I)$
is an indefinite one-dimensional elliptic differential operator
in the $y-$variable defined by
\begin{equation}\label{def_Lambda}
\Lambda(r)v:=-v''-\tfrac{(1+q\,\gamma(r,\cdot))\,s^2(r)}{\alpha^2\,q}\,v,
\quad\forall\,v\in H^2(I)\cap H^1_0(I),
\end{equation}
keeping in mind that the term
$u_r-\cunit\,\tfrac{\lambda}{q}\,u-\tfrac{{\dot
s}(r)}{s(r)}\,y\,u_y$  vanish at the endpoint of $I$.
In section 3 we provide some conditions on the data of the problem
that ensure invertibility of $\Lambda(r)$ for $r\in[0,R]$ along
with some regularity properties.
We note that the p.d.e. in \eqref{NewProblem2}, which follows
from (WA) after the aforementioned change of variable, is not a
usual Sobolev-type equation (cf. \cite{Lagnese}) like (WA) over a
horizontal bottom (cf. \cite{AkDZ1996}, \cite{Ak1999}).  Due to
the presence of the term $\frac{{\dot s}(t)}{s(t)}\,y\,u_y$, the
differential equation is of third order with respect to the space
variable $y$ and this offers an explanation of why an additional
boundary condition may be needed.
\par
In section 4, for the approximation to the problem
\eqref{NewProblem2} we propose and analyze a numerical method
that combines Crank-Nicolson time-stepping with a standard
second-order finite difference method in space. We would like to
stress that the convergence analysis of the proposed finite
difference scheme is not a repetition of the corresponding
analysis for the flat bottom case (cf. \cite{Ak1999},
\cite{AkDZ1996}). This is due to the fact that the differential
operator is third order with respect to $y$, which leads to a
truncation error of $O(1)$ at the nodes adjacent to the endpoints
of $I$. Building up a careful consistency argument is important
in proving a second-order error estimate.
\par
Finally, in section 5, we verify the accuracy and stability of
the finite difference scheme by means of numerical experiments and
apply it to solve an underwater acoustics problem in a downsloping
benchmark domain comparing the results with those obtained from
the model proposed in \cite{DSZ2009}.
\par
For rigorous error estimates on other finite difference and finite
element methods for the PE and WA equation on domains with
horizontal or variable bottom we refer the reader e.g. to the
papers \cite{Ak1999}, \cite{AkDZ1996}, \cite{ref8}, \cite{AnDZ},
\cite{AD}, and their references.
%
%
%
\section{A priori estimates}
\par
Our aim in this section is to establish some \textit{a priori}
estimates for the solution $v$ of the wide-angle equation (WA) on
$\mathcal{D}$ under some or all of the auxiliary conditions
considered in the previous section, and under several hypotheses
on the coefficients and the bottom. In what follows we shall
assume that the functions $v$, $\beta$, $s$, and $v_0$ are
sufficiently smooth so that the various estimates are valid.
%
%
Also, in our analysis we shall employ $L^2$, $H^1$ and $H^2$
range-dependent norms on $[0,s(r)]$ which are given, respectively, by
$\|v\|:=(\int_0^{s(r)}|v(r,z)|^2dz)^{\frac{1}{2}}$,
$\|v\|_1:=(\|v\|^2+\|v_z\|^2)^{\frac{1}{2}}$ and
$\|v\|_2:=(\|v\|_1^2+\|v_{zz}\|^2)^{\frac{1}{2}}$.
\par
We begin by establishing some basic identities.
\begin{lemma}\label{ll1}
Let $q\not=0$ and ${\mathcal F}$ be defined by \eqref{d1.6}. If
$v$ satisfies \text{\rm (WA)}, \eqref{d1.1} and \eqref{d1.2},
then, for $0\leq r\leq R$, the following identities hold:
\begin{equation}\label{eq11}
\begin{split}
\tfrac{d}{dr}\|v(r,\cdot)\|^2=&\,\dot{s}(r)|v(r,s(r))|^2
+\tfrac{2}{\lambda}\,\text{\rm
Im}(q)\,\|v_r(r,\cdot)\|^2\\
&\,-\tfrac{2}{\lambda}\,\int_0^{s(r)}\text{\rm
Im}\left[\beta(r,z)\right]\,|{\mathcal F}(v;r,z)|^2\;dz
-\tfrac{2\,\alpha^2}{\lambda}\,\text{\rm Im}\left[\,{\mathcal
F}(v_z;r,s(r))
\,\overline{{\mathcal F(v;r,s(r))}}\,\right],\\
\end{split}
\end{equation}
\begin{equation}\label{eq22}
\begin{split}
\tfrac{d}{dr}\|v_z(r,\cdot)\|^2=&\,\dot{s}(r)\,|v_z(r,s(r))|^2
+\tfrac{2}{\alpha^2}\,\text{\rm
Re}\left[\,\int_0^{\ssy s(r)}\tfrac{1+q\,\beta(r,z)}{q}
\,v_r(r,z)\overline{v}(r,z)\;dz\,\right]
-2\,\lambda\,{\rm Im}(\tfrac{1}{q})\|v_z(r,\cdot)\|^2\\
&\,+2\,\text{\rm Re}\left[\tfrac{1}{q}\,
\,{\mathcal F}(v_z;r,s(r))\,\overline{v(r,s(r))}\,\right]
+\tfrac{2}{\alpha^2}\,\lambda\,\int_0^{\ssy s(r)}\text{\rm
Im}\left(\tfrac{\beta(r,z)}{q}\right)\,|v(r,z)|^2\;dz,
\end{split}
\end{equation}
\begin{equation}\label{eq33}
\begin{split}
\frac{d}{dr}\|v_{zz}(r,\cdot)\|^2=&\dot{s}(r)|v_{zz}(r,s(r))|^2
+\tfrac{2}{\alpha^2}\,{\rm Re}\left[\frac{1}{q}\,\int_0^{\ssy s(r)}
v_r(r,\cdot)\,{\overline {v_{zz}}}(r,z)\;dz\,\right]\\
&-2\,\lambda\,{\rm Im}\left(\tfrac{1}{q}\right)\|v_{zz}\|^2
-\tfrac{2}{\alpha^2}\,{\rm Re}\left[\tfrac{1}{q}\,\int_0^{\ssy s(r)}
\beta(r,z)\,{\mathcal F}(v;r,z)\;\overline{v_{zz}}(r,z)\;dz\right],
\end{split}
\end{equation}
and
\begin{equation}\label{eq44}
\begin{split}
\int_0^{\ssy s(r)}(1+q\,\beta(r,z))\,|v_r(r,z)|^2\;dz
=&\,q\,\alpha^2\,\|v_{rz}(r,\cdot)\|^2\\
&+\cunit\,\alpha^2\,\lambda\,\int_0^{\ssy s(r)}
\left[\tfrac{\beta(r,z)}{\alpha^2}\,
v(r,z)+v_{zz}(r,z)\right]\,\overline{v_r(r,z)}\;dz\\
&\,-q\,\alpha^2\,v_{rz}(r,s(r))\,\overline{v_r(r,s(r))}.
\end{split}
\end{equation}
\end{lemma}
\begin{proof}
We multiply equation (WA) using \eqref{d1.6} by $\overline{{\mathcal F}(v;r,\cdot)}$,
integrate by parts with respect to $z$ and take imaginary parts to
get \eqref{eq11}. We then multiply (WA) using \eqref{d1.6} by
$\overline{v(r,\cdot)}$, $\overline{v_{zz}(r,\cdot)}$, respectively, integrate by
parts, and take real parts, obtaining \eqref{eq22}, \eqref{eq33},
respectively. The last equality \eqref{eq44} follows, by
multiplying (WA) using \eqref{d1.6} by $\overline{v_{r}(r,\cdot)}$ and
integrating.
\end{proof}
\par
From \eqref{eq11} it follows that if $v$ satisfies (WA),
\eqref{d1.1}, \eqref{d1.2}, and \eqref{d1.3}, and $\beta$ and $q$
are real, then $v$ satisfies the $L^2$ conservation property
\eqref{d1.4} if and only if \eqref{d1.5} holds.

Before embarking on our study of the downsloping bottom problem,
we prove a $H^2$-stability result in the upsloping bottom case
assuming only the homogeneous Dirichlet boundary condition
\eqref{d1.3}, thus complementing the $H^1$-estimate of
\cite{DSZ2009}. The proof requires that
$\frac{1}{\alpha}\displaystyle{\max_{0\leq r\leq R}}s(r)$ be
sufficiently small, i.e. a `small frequency-shallow water'
assumption. In what follows, $c$ or $C$ will denote generic
positive constants, not necessarily having the same values in any
two places.
%
%
%
\begin{proposition}\label{pd2.1}
Suppose that $\dot{s}(r)\leq 0$ for $r\in[0,R]$, $q\neq 0$, and
$\frac{1}{\alpha}\displaystyle{\max_{0\leq r\leq R}}s(r)$ is
sufficiently small. Then, there exists a constant $c$ such that
any solution $v$ of the ibvp {\rm (WA)}, \eqref{d1.1}-\eqref{d1.3}
satisfies
\begin{equation}\label{d2.5}
\|v(r,\cdot)\|_2\leq c\|v_0\|_2,\quad 0\leq r\leq R.
\end{equation}
\end{proposition}
%
%
%
\begin{proof}
We first prove that for $r\in[0,R]$ we have
\begin{equation}\label{d2.6}
\|v_r(r,\cdot)\|+\|{\mathcal F}(v;r,\cdot)\|\leq\,c\,\|v(r,\cdot)\|_2.
\end{equation}
To see this, note that \eqref{eq44} gives for $0\leq r\leq R$
\begin{equation}\label{d2.7}
\begin{split}
\|v_{rz}(r,\cdot)\|^2\leq\,\tfrac{c}{\alpha^2}\,\|v_r(r,\cdot)\|^2
+c\,\|v_r(r,\cdot)\|\,\|v(r,\cdot)\|_2+|v_{rz}(r,s(r))|\,|v_r(r,s(r))|.
\end{split}
\end{equation}
Since $v_r(r,0)=0$, we have
\begin{equation}\label{d2.8}
\|v_r(r,\cdot)\|\leq\,s(r)\,\|v_{rz}(r,\cdot)\|,\quad 0\leq r\leq R.
\end{equation}
By differentiating with respect to $r$ the Dirichlet boundary condition
$v(r,s(r))=0$, we obtain by the trace inequality
\begin{equation*}
|v_r(r,s(r))|=|\dot{s}|\;|v_z(r,s(r))|\leq\,c\,\|v(r,\cdot)\|_2.
\end{equation*}
Therefore, \eqref{d2.7} gives for $0\leq r\leq R$
\begin{equation}\label{d2.9}
\|v_{rz}(r,\cdot)\|^2\leq\,\tfrac{c}{\alpha^2}\,\|v_r(r,\cdot)\|^2
+c\,\|v_{rz}(r,\cdot)\|\,\|v(r,\cdot)\|_2
+c\,|v_{rz}(r,s(r))|\,\|v(r,\cdot)\|_2.
\end{equation}
The equation (WA) solved for $v_{rzz}$ and \eqref{d2.8} yields now
\begin{equation*}
\begin{split}
\|v_{rzz}(r,\cdot)\|\leq&\,c\,\left(\,\|v_r(r,\cdot)\|+\|v(r,\cdot)\|_2\,\right)\\
\leq&\,c\,\left(\,\|v_{rz}(r,\cdot)\|+\|v(r,\cdot)\|_2\,\right).\\
\end{split}
\end{equation*}
Thus by Sobolev's inequality we obtain
\begin{equation*}
\begin{split}
|v_{rz}(r,s(r))|\leq&\,c\,\left(\,
\|v_{rz}(r,\cdot)\|+\|v_{rzz}(r,\cdot)\|\,\right)\\
\leq&\,c\,\left(\,\|v_{rz}(r,\cdot)\|+\|v(r,\cdot)\|_2\,\right).
\end{split}
\end{equation*}
Using this in \eqref{d2.9} and
using \eqref{d2.8} we obtain, for any $\varepsilon>0$ and $0\leq
r\leq R$
\begin{equation*}
\|v_{rz}(r,\cdot)\|^2\leq\,\tfrac{c}{\alpha^2}\,s^2(r)\,
\|v_{rz}(r,\cdot)\|^2+\varepsilon\,\|v_{rz}(r,\cdot)\|^2
+c_{\varepsilon}\,\|v(r,\cdot)\|_2^2.
\end{equation*}
Hence, if $\frac{1}{\alpha}\displaystyle{\max_{0\leq r\leq
R}}s(r)$ is sufficiently small, we see, using again \eqref{d2.8},
that $\|v_r(r,\cdot)\|\leq\,c\,\|v(r,\cdot)\|_2$ for $0\leq r\leq R$.
Hence, \eqref{d2.6} follows, since
\begin{equation*}
\begin{split}
\|{\mathcal F}(v;r,\cdot)\|\leq&\,c\,\left(\,\|v_r(r,\cdot)\|
+\|v(r,\cdot)\|\,\right)\\
\leq&\,c\,\|v(r,\cdot)\|_2.\\
\end{split}
\end{equation*}
Adding now \eqref{eq22} and \eqref{eq33} and using the fact that
$\dot{s}(r)\leq 0$, estimates from the proof of \eqref{d2.6}, and
the Poincar$\acute{e}$-Friedrichs inequality we obtain for $0\leq
r\leq R$ that
\begin{equation*}
\begin{split}
\frac{d}{dr}\left(\,\|v_z(r,\cdot)\|^2+\|v_{zz}(r,\cdot)\|^2\,\right)
\leq&\,c\,\left(\,\|v_r(r,\cdot)\|^2+\|v_z(r,\cdot)\|^2+\|v_{zz}(r,\cdot)\|^2
+\|{\mathcal F}(v;r,\cdot)\|^2+\|v(r,\cdot)\|^2\,\right)\\
\leq&\,c\,\left(\,\|v_z(r,\cdot)\|^2+\|v_{zz}(r,\cdot)\|^2\,\right).\\
\end{split}
\end{equation*}
Hence, from Gr{\"o}nwall's inequality we obtain
\begin{equation*}
\|v_z(r,\cdot)\|^2+\|v_{zz}(r,\cdot)\|^2\leq c{\Big
(}\|v_z(0,\cdot)\|^2+\|v_{zz}(0,\cdot)\|^2{\Big )},
\end{equation*}
from which our conclusion follows. \end{proof}
\par
The $H^2$-stability estimate \eqref{d2.5} implies of course the
uniqueness of solution of the ibvp (WA),
\eqref{d1.1}-\eqref{d1.3} under the hypotheses of Proposition
\ref{pd2.1}, and forms the basis for a well-posedness theory for
this problem.
\par
We turn now to the downsloping bottom case, with which we shall
be concerned for the sequel of this paper. We first establish the
following basic \textit{a priori} estimates.
\begin{thm}\label{td.2.1}
Supose that $\dot{s}(r)>0$ for $r\in[0,R]$,
$q\neq 0$, and that either ${\rm Im}(q^{-1}+\beta(r,z))$ is of one
sign in $\mathcal{D}$ or that
$\frac{1}{\alpha}\displaystyle{\max_{0\leq r\leq R}}s(r)$ is
sufficiently small. Then the ibvp (WA),
\eqref{d1.1}-\eqref{d1.3}, \eqref{d1.7} is $L^2$-, $H^1$-, and
$H^2$-stable.
\end{thm}
\begin{proof}
Let $v$ be the solution of the ibvp (WA),
\eqref{d1.1}-\eqref{d1.3}, \eqref{d1.7}. Then it follows from
\eqref{d1.6} that $f(r,z):={\mathcal F}(v;r,z)$ satisfies
\begin{equation}\label{d26}
\begin{split}
&\left(\tfrac{1}{q}+\beta\right)\,f+\alpha^2\,f_{zz}
=-\cunit\,\tfrac{\lambda}{q}\,v,\quad (r,z)\in\mathcal{D},\\
&f(r,0)=f(r,s(r))=0,\quad 0\leq r\leq R.
\end{split}
\end{equation}
For each $r\in [0,R]$ consider the operator $\mathcal{L}(r):H^2\cap
H_0^1\rightarrow L^2$ defined by
\begin{equation*}
\mathcal{L}(r)u:=\left(\tfrac{1}{q}+\beta(r,\cdot)\right)\,u
+\alpha^2\,\partial_z^2u,\quad\forall \,u\in H^2\cap H_0^1.
\end{equation*}
Then, under our hypotheses,
$\mathcal{L}(r)$ is invertible in $H^2\cap H_0^1$, in the sense that
$\mathcal{L}(r)u=0$ implies $u=0$. To see this, note that from
$\mathcal{L}u=0$ for $u\in H^2\cap H_0^1$, it follows that
$\int_0^{s(r)}\mathcal{L}u\, \overline{u}dz=0$, and by
integration by parts, that
\begin{equation}\label{d27}
\int_0^{s(r)}\Big{[}(q^{-1}+\beta)|u|^2-\alpha^2|u_z|^2\Big{]}dz=0.
\end{equation}

Taking real parts in \eqref{d27} and using the fact that
$\|u\|\leq s(r)\|u_z\|$, it follows that
\begin{equation*}
\|u_z(r,\cdot)\|^2\,\left(1-\tfrac{1}{\alpha^2}\,\,\max_{0\leq r\leq R}s^2(r)
\,\,\max_{z\in\mathcal{D}}|{\rm Re}(q^{-1}+\beta)|\right)\leq 0.
\end{equation*}
Hence, if $\frac{1}{\alpha}\displaystyle{\max_{0\leq r\leq
R}}s(r)$ is sufficiently small we have that $u=0$. Alternatively,
taking imaginary parts in \eqref{d27} we obtain that
$\int_0^{s(r)}{\rm Im}(q^{-1}+\beta)|u|^2dz=0$, from which if
${\rm Im}(q^{-1}+\beta)$ is nonzero and of one sign in
$\mathcal{D}$, then $u=0$.

We conclude, from standard elliptic p.d.e. theory and the Fredholm
alternative, \cite{Evans}, that given $w\in L^2$, then
$\mathcal{L}^{-1}w\in H^2\cap H_0^1$ and
$\|\mathcal{L}^{-1}w\|_2\leq C\|w\|$, for some positive $C=C(r)$,
independent of $w$. Since the coefficients of $\mathcal{L}$ are
smooth, $C$ may be taken as a continuous function on $[0,R]$.
Applying this result to the bvp \eqref{d26} we see that for
$0\leq r\leq R$,
\begin{equation}\label{d28}
\|f(r,\cdot)\|_2\leq\,C\,\|v(r,\cdot)\|.
\end{equation}
By the definition of $f$, it follows that on $[0,R]$
\begin{equation}\label{d29}
\|v_r(r,\cdot)\|_{\ell}\leq\,C\,\|v(r,\cdot)\|_{\ell},\quad\ell=0,1,2.
\end{equation}
Hence, by \eqref{eq11}, $L^2$-stability of the ibvp under
consideration follows. To get the $H^1$-estimate, we use
\eqref{eq22}, \eqref{d29} and the fact that $v_z(z,s(r))=0$ for
$0\leq r\leq R$, as remarked in the Introduction. It follows that
for $0\leq r\leq R$
\begin{equation*}
\tfrac{d}{dr}\|v_z(r,\cdot)\|^2\leq\,C\,\left(\|v(r,\cdot)\|^2
+\|v_z(r,\cdot)\|^2\right),
\end{equation*}
from which $H^1$-stability follows by the
Poincar$\acute{e}$-Friedrichs inequality. Finally, to get
$H^2$-stability, note that \eqref{eq33}, \eqref{d28}, \eqref{d29}
yield for $0\leq r\leq R$
\begin{equation}\label{d30}
\tfrac{d}{dr}\|v_{zz}\|^2\leq
\dot{s}(r)|v_{zz}(r,s(r))|^2+C\,\left(\|v(r,\cdot)\|^2
+\|v_{zz}(r,\cdot)\|^2\right).
\end{equation}
From $v_z(r,s(r))=0$ for $0\leq r\leq R$, it follows that
$v_{zr}(r,s(r))+\dot{s}(r)v_{zz}(r,s(r))=0$. Since,
$v_{zr}(r,s(r))=\frac{1}{q}\,f(r,s(r))$, it follows by the trace
inequality and \eqref{d28} that for $0\leq r\leq R$
\begin{equation*}
\begin{split}
\dot{s}(r)\,|v_{zz}(r,s(r))|^2=&\,\tfrac{1}{|\dot{s}(r)|}\,|v_{zr}(r,s(r))|^2\\
\leq&\,C\,\|f(r,\cdot)\|_2^2\\
\leq&\,C\,\|v(r,\cdot)\|^2.
\end{split}
\end{equation*}
From \eqref{d30} the $H^2$-stability estimate follows now.
\end{proof}
\begin{rem} If ${\rm Im}(\beta)\geq 0$ and ${\rm Im}(q)<0$ or if ${\rm
Im}(\beta)>0$ and $q$ is a real, nonzero constant, then ${\rm
Im}\left(\frac{1}{q}+\beta\right)>0$ follows. The ibvp {\rm (WA)},
\eqref{d1.1}-\eqref{d1.3}, \eqref{d1.7} is of course
$L^2$-conservative if $q$ and $\beta$ are real.
\end{rem}
\section{Invertibility conditions for $\Lambda$}
%
%
%
%
%
%
%
%
%
%
%
%
%
%
%
Assuming that the operator $\Lambda(r)$ defined by \eqref{def_Lambda} is invertible for all
$r\in[0,R]$ and that $u_0\in H_0^1(I)$, we may write the problem
\eqref{NewProblem2} equivalently as:
\begin{equation}\label{NewProblem3}
\begin{split}
&u_r-\tfrac{{\dot s}(r)}{s(r)}\,y\,u_y=
\cunit\,\tfrac{\lambda\,s^2(r)}{\alpha^2\,q^2}\,\,T(r)u
+\cunit\,\tfrac{\lambda}{q}\,u
\quad\forall\,(r,y)\in\Omega,\\
&u(r,1)=0\quad\forall\,r\in[0,R],\\
&u(0,y)=u_0(y)\quad\forall\,y\in I,\\
\end{split}
\end{equation}
where $T(r):=\Lambda^{\ssy -1}(r)$ for $r\in[0,R]$. To simplify
the notation we set
\begin{equation*}
\delta(r):=\tfrac{{\dot s}(r)}{s(r)}, \quad
\zeta(r,y):=\tfrac{(1+q\,\gamma(r,y))\,s^2(r)}{\alpha^2\,q},
\quad\xi(r):=\tfrac{\lambda\,s^2(r)}{\alpha^2\,q^2}
\end{equation*}
for $r\in[0,R]$ and $y\in I$.
%
%
\begin{remark}
Let $g(r,y):=u_r(r,y)-\delta(r)\,y\,u_y(r,y)$. Assuming that the
solution of \eqref{NewProblem3} is smooth on $\Omega$, we obtain
the compatibility conditions $g(r,1)=g(r,0)=0$, which yield that
$u_r(r,0)=0$ and $u_y(r,1)=0$ for $r\in[0,R]$. Finally, since
$u(0,0)=u_0(0)=0$, we get the surface pressure release condition
$u(r,0)=0$ for $r\in[0,R]$.
\end{remark}
\begin{remark}\label{rd3.2}
Whereas the ibvp (WA), \eqref{d1.1}-\eqref{d1.3}, \eqref{d1.7} in
the $r$, $z$ variables is $L^2$-conservative, in the sense that
its solution satisfies \eqref{d1.4} for $q$, $\beta$ real, the
solution of the transformed ibvp \eqref{NewProblem2} in the $r$,
$y$ variables conserves the quantity $s(r)\|u(r,\cdot)\|_{\ssy
0,I}^2$ for $q$, $\gamma$ real.
\end{remark}

In the rest of this section we present some conditions on the data
that ensure invertibility of the operator $\Lambda$. The
conditions are similar in nature to those in the hypothesis of
Theorem \ref{td.2.1}, but it is useful to present them here
because they are expressed in terms of the $y$ variable and
motivate analogous sufficient conditions for the invertibility of
the discrete operators in the next section. We use the notation
$\|v\|_{\ssy j,I},\;j\geq 0$, for the norm in the Sobolev space
$H^j(I)$ and put $|v|_{\ssy 1,I}=\|v_y\|_{\ssy 0,I}$.
%
%
%
%
%
%
%
%
\begin{lemma}
Assume that
\begin{equation}\label{COE_2}
C_{\ssy E\!B}:=\inf_{\ssy\Omega}\left[\tfrac{1}{(C_{\ssy
P\!F})^2}-\tfrac{s^2}{\alpha^2\,|q|^2}\,\left(\text{\rm
Re}(q)+|q|^2\,\text{\rm Re}(\gamma)\right)\right]>0,
\end{equation}
where $C_{\ssy P\!F}>0$ is the constant of the
Poincar{\'e}-Friedrichs inequality on $I$, i.e., $\|v\|_{\ssy
0,I}\leq\,C_{\ssy P\!F}\,|v|_{\ssy 1,I}$ for $v\in H_0^1(I)$, or
that there exists $\delta_{\star}\in\{1,-1\}$ such that
\begin{equation}\label{CEE_1}
C_{\ssy B\!B}:=\inf_{\ssy\Omega}\left[\,
\tfrac{\delta_{\star}\,s^2}{\alpha^2\,|q|^2}\,\left(\text{\rm
Im}(q)-|q|^2\,\text{\rm Im}(\gamma)\right)\,\right]>0.
\end{equation}
Then, there exists a constant $C>0$ such that
\begin{equation}\label{COE_3}
\max_{r\in[0,R]}\|T(r)\psi\|_{\ssy 1,I}\leq\,C\,\|\psi\|_{\ssy
0,I},\quad\forall\,\psi\in L^2(I).
\end{equation}
\end{lemma}
%
%
%
%
%
%
%
%
%
%
%
\begin{proof}
Let $r\in[0,R]$, $v\in H^2(I)\cap H_0^1(I)$ and $\psi\in L^2(I)$.
First, observe that
\begin{equation}\label{Basic_Tool_1}
\text{\rm Re}(\Lambda(r)v,v)_{\ssy 0,I}=\|v'\|^2_{\ssy
0,I}-\int_{\ssy I}\tfrac{\text{\rm Re}({\overline q}
+|q|^2\,\gamma)\,s^2(r)}{\alpha^2\,|q|^2}\,|v|^2\;dy.
\end{equation}
\par
When \eqref{COE_2} holds, use of the Poincar{\'e}-Friedrichs
inequality and \eqref{Basic_Tool_1} gives
\begin{equation}\label{COE_5}
\begin{split}
\text{\rm Re}(\Lambda(r)v,v)_{\ssy 0,I}\ge\,&\int_{\ssy
I}\left[\tfrac{1}{(C_{\ssy P\!F})^2}-\tfrac{\text{\rm
Re}(q+|q|^2\,\gamma)\,s^2(r)}
{\alpha^2\,|q|^2}\right]\,|v|^2\;dy\\
\ge\,&C_{\ssy E\!B}\,\|v\|^2_{\ssy 0,I}.
\end{split}
\end{equation}
\par
When \eqref{CEE_1} holds, we have
\begin{equation}\label{COE_555}
\begin{split}
\big|\,\text{\rm Im}(\Lambda(r)v,v)_{\ssy 0,I}\,\big|
\ge&\,\delta_{\star}\,\text{\rm Im}(\Lambda(r)v,v)_{\ssy 0,I}\\
\ge&\,\delta_{\star}\,\int_{\ssy I}\tfrac{\text{\rm Im}(q
+|q|^2\,{\overline \gamma})\,s^2(r)}{\alpha^2\,|q|^2}\,|v|^2\;dx\\
\ge&\,C_{\ssy B\!B}\,\|v\|^2_{\ssy 0,I}.
\end{split}
\end{equation}
\par
Since, by definition, $T(r)\psi=\Lambda^{\ssy -1}(r)\psi\in
H^2(I)\cap H_0^1(I)$, we set $v=T(r)\psi$ in \eqref{COE_5} or
\eqref{COE_555} and apply the Cauchy-Schwarz inequality to
conclude for $r\in[0,R]$
\begin{equation}\label{COE_4}
\|T(r)\psi\|_{\ssy 0,I}\leq\,C\,\|\psi\|_{\ssy 0,I}.
\end{equation}
Then, we combine \eqref{Basic_Tool_1} for $v=T(r)\psi$, and
\eqref{COE_4} to get
\begin{equation}\label{COE_7}
\begin{split}
|T(r)\psi|_{\ssy 1,I}^2=&\,\text{\rm Re}(\psi,T(r)\psi)_{\ssy
0,I}+\int_{\ssy I}\tfrac{\text{\rm Re}({\overline q}
+|q|^2\,\gamma)\,s^2(r)}{\alpha^2\,|q|^2}\,|T(r)\psi|^2\;dy\\
\leq&\,\|\psi\|_{\ssy 0,I}\,\|T(r)\psi\|_{\ssy 0,I}
+C\,\|T(r)\psi\|^2_{\ssy 0,I}\\
\leq&\,C\,\|\psi\|_{\ssy 0,I}^2.\\
\end{split}
\end{equation}
Now \eqref{COE_3} follows easily from \eqref{COE_7} and
\eqref{COE_4}.
\end{proof}
%
%
%
%
%
%
%
%
%
%
%
%
%
Assuming that \eqref{COE_2} or \eqref{CEE_1} holds and using a
induction argument, it is easy to establish that for
$m\in\nset_0$, there exists $C_m>0$ such that
\begin{equation}\label{Ringer_1}
\max_{r\in[0,R]}|T(r)f|_{m+2,\infty,{\ssy I}}
\leq\,C_m\,|f|_{m,\infty,{\ssy I}},\quad\forall\,f\in
C^{m}(I;\cset),
\end{equation}
where on $C^m$, $|f|_{m,\infty,{\ssy I}}=\displaystyle{\max_{y\in
I}}|f^{(m)}(y)|$. In addition, for $m\in\nset_0$ and
$\ell\in\nset$, there exists $C_{\ell,m}>0$ such that
\begin{equation}\label{Cagaroo_1}
\max_{r\in[0,R]}|\partial_r^{\ell}(T(r)f)|_{m+2,\infty,{\ssy I}}
\leq\,C_{\ell,m}\,|f|_{\max\{m-2,0\},\infty,{\ssy
I}},\quad\forall\,f\in C^{\max\{m-2,0\}}(I;\cset).
\end{equation}
%
%
%
%
\begin{remark}\label{rd3.3}
Differentiating both sides of the p.d.e. in \eqref{NewProblem3}
once with respect to $y$, taking $y=1$ and using the boundary
conditions we get
\begin{equation*}
u_{yy}(r,1)=-\cunit \,\tfrac{\lambda\,s^3(r)}{\alpha^2\,q^2\,{\dot
s}(r)}\,(T(r)u)_y(r,1), \quad\forall\,r\in[0,R].
\end{equation*}
This reminds us in some sense of the boundary condition proposed
in \cite{DSZ2009} for downsloping bottoms, which  has the form
\begin{equation*}
u_{yy}(r,1)=\cunit\,\tfrac{2}{\alpha}\,s(r)\,{\dot
s}(r)\,u_y(r,1),\quad\forall\,r\in[0,R].
\end{equation*}
\end{remark}
\begin{remark}\label{rd3.4}
Condition \eqref{COE_2} follows from the hypothesis that
$\frac{s}{\alpha}$ is sufficiently small for $0\leq r\leq R$. If ${\rm
Im}(\gamma)\geq 0$ in $\Omega$ and ${\rm Im}(q)<0$ or ${\rm
Im}(\gamma)>0$ in $\Omega$ and $q$ is real and nonzero, condition
\eqref{CEE_1} is valid for $\delta_*=-1$.
\end{remark}
%
%
%
%
%
\section{A Finite difference method}\label{sec3}
%
%
In this section we construct and analyze a finite difference
method for approximating the solution of the ibvp
\eqref{NewProblem3}.
\subsection{Notation and preliminaries}
%
Let $J\in{\mathbb N}$ with $J\ge3$. We introduce a partition of
$[0,1]$ with width $h:=\frac{1}{J}$ and nodes $y_j:=j\,h$ for
$j=0,\dots,J$. Taking into account the homogeneous Dirichlet
boundary conditions at the endpoints of $I$ we define the space
$X_h$ of the finite difference approximations by
\begin{equation*}
X_h:=\{(v_j)_{j=0}^{\ssy J}\in{\mathbb C}^{\ssy
J+1}:\,\,\,v_0=v_{\ssy J}=0\}.
\end{equation*}
On $X_h$ we define a discrete $L^2(I)$ norm $\|\cdot\|_{0,h}$
given by $\|v\|_{0,h}:=\left(h\sum_{j=1}^{\ssy
J-1}|v_j|^2\right)^{\frac{1}{2}}$ for $v\in X_h$, which is
produced by the inner product $(\cdot,\cdot)_{0,h}$ defined by
$(v,w)_{0,h}:=h\sum_{j=1}^{\ssy J-1}v_j\,{\overline{w_j}}$ for
$v$, $w\in X_h$. Also, on $X_h$ we define a discrete $H^1(I)$
seminorm $|\cdot|_{1,h}$ by $|v|_{1,h}:=\left(h\sum_{j=0}^{\ssy
J-1}\left|\frac{v_{j+1}-v_{j}}{h}\right|^2\right)^{\frac{1}{2}}$
for $v\in X_h$, a discrete $H^1(I)$ norm $\|\cdot\|_{1,h}$ by
$\|v\|_{1,h}:=\left[\|v\|^2_{0,h}+|v|^2_{1,h}\right]^{\frac{1}{2}}$
for $v\in X_h$, and a discrete $L^{\infty}(I)$ norm
$|\cdot|_{\infty,h}$ by $|v|_{\infty,h}:=\max_{1\leq{j}\leq{\ssy
J-1}}|v_j|$ for $v\in X_h$.
\par
We define further the second-order difference operator
$\Delta_h:X_h\rightarrow X_h$ by
\begin{equation*}
(\Delta_hv)_j=\tfrac{v_{j-1}-2\,v_j+v_{j+1}}{h^2},\quad
j=1,\dots,J-1,\quad\forall\,v\in X_h,
\end{equation*}
and for $r\in[0,R]$, the discrete elliptic operator
$\Lambda_h(r):X_h\rightarrow X_h$ by
\begin{equation*}
(\Lambda_h(r)v)_j:=-(\Delta_hv)_j
-s^2(r)\,\tfrac{1+q\,\gamma(r,y_j)}{\alpha^2\,q}\,v_j,\quad
j=1,\dots,J-1,\quad\forall\,v\in X_h.
\end{equation*}
Also, we define the first-order difference operator
$\partial_h:X_h\rightarrow X_h$ by
\begin{equation*}
(\partial_hv)_j:=\tfrac{v_{j+1}-v_{j-1}}{2h},\quad j=1,\dots,J-1,
\quad\forall\,v\in X_h,
\end{equation*}
%
%
%
%
%
%
%
%
and introduce the auxiliary operators $I_h:X_h\rightarrow X_h$
given by $(I_hv)_j:=\tfrac{v_{j+1}+v_{j-1}}{2}$ for
$j=1,\dots,J-1$ and $v\in X_h$, and $\otimes:X_h^2\rightarrow X_h$
defined by $(v\otimes w)_j=v_j\,w_j$ for $j=1,\dots,J-1$ and $v$,
$w\in X_h$. Also, we let $\omega\in X_h$ be such that
$\omega_j:=y_j$ for $j=0,\dots,J-1$. For
$f:[0,1]\rightarrow{\mathbb C}$ we define $P_hf\in X_h$ by
$(P_hf)_j:=f(y_j)$ for $j=1,\dots,J-1$.
\par
It is easy to show that
\begin{equation}\label{Discrete_PF}
\|v\|_{0,h}\leq\,\tfrac{\sqrt{2}}{2}\,|v|_{1,h} \quad\forall\,v\in
X_h,
\end{equation}
\begin{equation}\label{Max_H1_Dominus}
|v|_{\infty,h}\leq\,|v|_{1,h}\quad\forall\,v\in X_h,
\end{equation}
\begin{equation}\label{Max_L2_Inverse}
|v|_{\infty,h}\leq\,h^{-\frac{1}{2}}\,\|v\|_{0,h}\quad\forall\,v\in
X_h,
\end{equation}
and
\begin{equation}\label{integration_by_parts}
(\Delta_hv,v)_{0,h}=-|v|_{1,h}^2\quad\forall\,v\in X_h.
\end{equation}
\par
Let $N\in{\mathbb N}$. We define a partition of the range
interval $[0,R]$ with nodes $(r^n)_{n=0}^{\ssy N}$ given  by
$r^n:=n\,k$ for $n=0,\dots,N$, and let $(u^n)_{n=0}^{\ssy
N}\subset X_h$ be such that $(u^n)_j:=u(r^n,y_j)$ for
$j=1,\dots,J-1$ and $n=0,\dots,N$, where $u$ is the solution of
\eqref{NewProblem2}. Also, we set
$r^{n+\frac{1}{2}}:=\frac{r^{n+1}+r^n}{2}$ for $n=0,\dots,N-1$.
%
%
%
%
%
%
%
%
%
\begin{lemma}
For $v\in X_h$ we have
\begin{equation}\label{ErrEstim3}
\text{\rm
Re}(\omega\otimes\partial_hv,v)_{0,h}=-\tfrac{1}{2}\,(v,I_hv)_{0,h}.
\end{equation}
\end{lemma}
%
%
%
%
%
%
%
%
\begin{proof}
For $v\in X_h$, we have
\begin{equation*}
\begin{split}
(\omega\otimes
\partial_hv,v)_{0,h}=&\,\tfrac{1}{2}
\,\,\sum_{j=1}^{\ssy J-1}y_{j-1}\,\,{\overline
{v_{j-1}}}\,\,v_{j} -\tfrac{1}{2}\,\sum_{j=1}^{\ssy J-1}
y_{j+1}\,\,{\overline{v_{j+1}}}\,\,v_{j}\\
=&\,\tfrac{1}{2}\,\sum_{j=1}^{\ssy
J-1}(y_{j-1}-y_j)\,v_{j}\,{\overline
{v_{j-1}}}-\tfrac{1}{2}\,\sum_{j=1}^{\ssy
J-1}y_j\,v_j\,{\overline{(v_{j+1}-v_{j-1})}}
-\tfrac{1}{2}\,\sum_{j=1}^{\ssy J-1}
(y_{j+1}-y_j)\,v_j\,{\overline{v_{j+1}}}\\
=&\,-{\overline{(\omega\otimes
\partial_hv,v)_{0,h}}}-h\,\sum_{j=1}^{\ssy J-1}v_j\,
\tfrac{{\overline{v_{j+1}+v_{j-1}}}}{2},\\
\end{split}
\end{equation*}
which easily yields \eqref{ErrEstim3}.
\end{proof}
%
%
%
%
%
%
%
%
%
\begin{lemma}
For $v\in X_h$ we have
\begin{equation}\label{Iden_Bound_1}
\|I_hv\|_{\ssy 0,h}\leq\,\|v\|_{0,h}.
\end{equation}
\end{lemma}
%
%
%
%
%
%
%
%
\begin{proof}
For $v\in X_h$, we have
\begin{equation*}
\begin{split}
\|I_hv\|_{0,h}=&\,\tfrac{1}{2}\,\left(h\sum_{j=1}^{\ssy
J-1}\left|v_{j+1}+v_{j-1}\right|^2\right)^{\frac{1}{2}}\\
\leq&\,\tfrac{1}{2}\,\left(h\sum_{j=1}^{\ssy
J-1}|v_{j-1}|^2\right)^{\frac{1}{2}}
+\tfrac{1}{2}\,\left(h\sum_{j=1}^{\ssy
J-1}|v_{j+1}|^2\right)^{\frac{1}{2}}\\
\leq&\,\tfrac{1}{2}\,\left(h\sum_{j=0}^{\ssy
J-2}|v_{j}|^2\right)^{\frac{1}{2}}
+\tfrac{1}{2}\,\left(h\sum_{j=2}^{\ssy
J}|v_{j}|^2\right)^{\frac{1}{2}},\\
\end{split}
\end{equation*}
which easily yields \eqref{Iden_Bound_1}.
\end{proof}
%
%
%
%
%
%
%
\begin{lemma}
For $v\in X_h$ we have
\begin{equation}\label{DSobolev_Bound}
|v|_{\infty,h}\leq\,\sqrt{2}\,\left(\,\|v\|_{0,h}
+\|\partial_hv\|_{0,h}\right).
\end{equation}
\end{lemma}
%
%
%
%
%
%
\begin{proof}
Let $v\in X_h$ with $v\not=0$. Then, there exists
$i_0\in\{1,\dots,J-1\}$ such that $|v|_{\infty,h}=|v_{i_0}|$. If
$i_0$ is even, i.e. $i_0=2\,m_0$ for some $m_0\in\nset$, then we
have
\begin{equation}\label{aivali_1}
\begin{split}
|v_{i_0}|=&\,|v_{2m_0}|\\
\leq&\,2\,\sum_{\ell=0}^{m_0-1}h\,\left|
\tfrac{v_{2(\ell+1)}-v_{2\ell}}{2h}\right|\\
\leq&\,2\,\left(h\sum_{\ell=0}^{m_0-1}1\right)^{\frac{1}{2}}\,
\left(h\sum_{\ell=0}^{m_0-1}\left|
\tfrac{v_{2(\ell+1)}-v_{2\ell}}{2h}\right|^2\right)^{\frac{1}{2}}\\
\leq&\,\sqrt{2}\,\,\|\partial_hv\|_{0,h}.\\
\end{split}
\end{equation}
Let us now assume that $i_0$ is odd, i.e. $i_0=2\,m_0-1$ for some
$m_0\in\nset$. If $J$ is odd, i.e. $J=2\,J_{\star}+1$ for some
$J_{\star}\in\nset$, then
\begin{equation}\label{aivali_2}
\begin{split}
|v_{i_0}|=&\,|v_{2m_0-1}|\\
\leq&\,2\,h\,\sum_{\ell=m_0}^{\ssy J_{\star}}\left|
\tfrac{v_{2\ell+1}-v_{2\ell-1}}{2h}\right|\\
\leq&\,2\,\left(h\sum_{\ell=m_0}^{\ssy
J_{\star}}1\right)^{\frac{1}{2}}\, \left(h\sum_{\ell=m_0}^{\ssy
J_{\star}}\left|
\tfrac{v_{2(\ell+1)}-v_{2\ell}}{2h}\right|^2\right)^{\frac{1}{2}}\\
\leq&\,\sqrt{2}\,\,\|\partial_hv\|_{0,h}.\\
\end{split}
\end{equation}
Now, we assume that $J$ is even, i.e. $J=2\,J_{\star}$ for some
$J_{\star}\in\nset$. We define $(w_{\ell})_{\ell=1}^{\ssy
J_{\star}}\subset\cset$, by $w_{\ell}=v_{2\ell-1}$ for
$\ell=1,\dots,J_{\star}$. Also, let
$\ell_{\star}\in\{1,\dots,J_{\star}\}$ such that
$|w_{\ell_{\star}}|=\min_{1\leq\ell\leq{J_{\star}}}|w_{\ell}|$.
Then, we have
\begin{equation}\label{aivali_3}
\begin{split}
|w_{m_0}|\leq&\,|w_{\ell_{\star}}|+\sum_{\ell=1}^{\ssy
J_{\star}-1}|w_{\ell+1}-w_{\ell}|\\
\leq&\,2\,h\,\sum_{\ell=1}^{\ssy
J_{\star}}|w_{\ell}|+2\,h\,\sum_{\ell=1}^{\ssy
J_{\star}-1}\left|\tfrac{w_{\ell+1}-w_{\ell}}{2h}
\right|\\
\leq&\,\sqrt{2}\,\|v\|_{0,h}+\sqrt{2}\,\|\partial_hv\|_{0,h}.
\end{split}
\end{equation}
Thus, \eqref{DSobolev_Bound} follows from \eqref{aivali_1},
\eqref{aivali_2} and \eqref{aivali_3}.
\end{proof}
%
%
%
%
%
%
%
\subsection{Properties of the discrete elliptic operator
$\Lambda_h$ and its inverse}\label{sd4.2}
\begin{lemma}
We assume that
\begin{equation}\label{DCE_40}
C_{\ssy D\!E\!B}:=\inf_{\ssy\Omega}\left[\tfrac{1}{(C_{\ssy
D\!P\!F})^2}-\tfrac{s^2}{\alpha^2\,|q|^2}\,\left(\text{\rm
Re}(q)+|q|^2\,\text{\rm Re}(\gamma)\right)\right]>0,
\end{equation}
where $C_{\ssy D\!P\!F}\in(0,\tfrac{\sqrt{2}}{2}]$ is the optimal
constant in \eqref{Discrete_PF}, or,
that there exists $\delta_{\star}\in\{1,-1\}$ such that
\begin{equation}\label{DCE_60}
C_{\ssy D\!B\!B}:=\inf_{\ssy\Omega}\left[\,
\tfrac{\delta_{\star}\,s^2}{\alpha^2\,|q|^2}\,\left(\text{\rm
Im}(q)-|q|^2\,\text{\rm Im}(\gamma)\right)\,\right]>0.
\end{equation}
Then, we have
\begin{equation}\label{DCE_42}
\max_{r\in[0,R]}\|T_h(r)v\|_{1,h} \leq\,C\,\|v\|_{0,h}
\quad\forall\,v\in X_h,
\end{equation}
where $T_h(r):=\Lambda_h^{-1}(r)$ for $r\in[0,R]$.
\end{lemma}
%
%
%
%
%
%
%
\begin{proof}
Let $r\in[0,R]$ and $v\in X_h$. Then, we have
\begin{equation}\label{DCE_50}
\text{\rm Re}(\Lambda_h(r)v,v)_{\ssy
0,h}=|v|_{1,h}^2-\tfrac{s^2(r)}{\alpha^2\,|q|^2}\,h\,\sum_{j=1}^{\ssy
J-1}\left[\text{\rm Re}(q)+|q|^2\,\text{\rm
Re}(\gamma(r,y_j))\right]\,|v_j|^2.
\end{equation}
\par
When \eqref{DCE_40} holds, then \eqref{DCE_50} and
\eqref{Discrete_PF} yield
\begin{equation}\label{DCE_41}
\begin{split}
\text{\rm Re}(\Lambda_h(r)v,v)_{\ssy 0,h}\ge&\,h\sum_{j=1}^{\ssy
J-1}\left[\tfrac{1}{(C_{\ssy
D\!P\!F})^2}-\tfrac{s^2(r)}{\alpha^2\,|q|^2}\, (\text{\rm
Re}(q)+|q|^2\,\text{\rm Re}(\gamma(r,y_j)))\,\right]\,\,|v_j|^2\\
\ge&\,C_{\ssy D\!E\!B}\,\|v\|_{0,h}^2 \quad\forall\,v\in
X_h,\quad\forall\,r\in[0,R].\\
\end{split}
\end{equation}
\par
When \eqref{DCE_60} holds, then, we have
\begin{equation}\label{DCE_70}
\begin{split}
\big|\,\text{\rm Im}(\Lambda_h(r)v,v)_{\ssy 0,h}\,\big|\ge&\,
\delta_{\star}\,\text{\rm Im}(\Lambda_h(r)v,v)_{\ssy
0,h}\\
\ge&\,\delta_{\star}\,\tfrac{s^2(r)}{\alpha^2\,|q|^2}\,h\,\sum_{j=1}^{\ssy
J-1}\left[\text{\rm Im}(q)-|q|^2\,\text{\rm
Im}(\gamma(r,y_j))\right]\,|v_j|^2\\
\ge&\,C_{\ssy D\!B\!B}\,\|v\|_{0,h}^2\quad\forall\,v\in
X_h,\quad\forall\,r\in[0,R].\\
\end{split}
\end{equation}
Now, using \eqref{DCE_41} or \eqref{DCE_70}, and the
Cauchy-Schwarz inequality we arrive at
\begin{equation}\label{DCE_51}
\|T_h(r)v\|_{0,h}\leq\,C\,\|v\|_{0,h}.
\end{equation}
Next, we use \eqref{DCE_50} and \eqref{DCE_51} to get
\begin{equation}\label{DCE_52}
\begin{split}
|T_h(r)v|_{1,h}^2\leq&\,C\,\left(\,\|T_h(r)v\|^2_{0,h}
+\|v\|_{0,h}\,\|T_h(r)v\|_{0,h}\,\right)\\
\leq&\,C\,\|v\|_{0,h}^2.\\
\end{split}
\end{equation}
Thus, \eqref{DCE_42} follows easily from \eqref{DCE_51} and
\eqref{DCE_52}.
\end{proof}
%
%
%
%
%
%
%
%
%
%
%
%
%
\begin{proposition}\label{LHMMA_2}
We assume that \eqref{DCE_40} or \eqref{DCE_60} hold.
Then, there exist positive constants $C_{\ssy A}$ and $C_{\ssy
B}$ such that
\begin{equation}\label{DCoe_10}
\|T_h(r)P_h\phi-P_hT(r)\phi\|_{1,h} \leq\,C_{\ssy A}\,h^2\,
|\phi|_{2,\infty,{\ssy I}}, \quad\forall\,\phi\in C^2(I;{\mathbb
C}),\quad\forall\,r\in[0,R],
\end{equation}
and
\begin{equation}\label{DCoe_1000}
\begin{split}
\left\|\,[T_h(r)P_h\phi-T_h(\tau)P_h\phi]-[P_hT(r)\phi-P_hT(\tau)\phi]\,\right\|_{1,h}
\leq&\,C_{\ssy B}\,\Big[\,h^3\,\,|\phi|_{3,\infty,{\ssy I}}
+h^2\,|r-\tau|\,|\phi|_{2,\infty,{\ssy I}}\,\Big]\\
&\quad\quad\quad\forall\,\phi\in C^3(I;{\mathbb C}),
\quad\forall\,r,\tau\in[0,R].\\
\end{split}
\end{equation}
\end{proposition}
%
%
%
%
%
%
%
\begin{proof}
Let $r$, $\tau\in[0,R]$, $\phi\in C^2(I;{\mathbb C})$,
$\psi(r,y):=(T(r)\phi)(y)$ and $E(r)\in X_h$ defined by
\begin{equation*}
E(r):=T_h(r)P_h\phi-P_hT(r)\phi.
\end{equation*}
Then, we have that
\begin{equation}\label{DCoe_11}
\Lambda_h(r)E(r)=\eta(r)
\end{equation}
where $\eta(r)\in X_h$ with
\begin{equation*}
\begin{split}
(\eta(r))_j=&\,-\left[\,\psi_{yy}(r,y_j)
-\tfrac{\psi(r,y_{j-1})-2\psi(r,y_j)+\psi(r,y_{j+1})}{h^2}
\right]\\
=&\,\tfrac{h^2}{12}\,\psi_{yyyy}(r,\xi_j(r)),\quad
j=1,\dots,J-1,\\
\end{split}
\end{equation*}
with $\xi_j(r)\in(y_{j-1},y_{j+1})$ for $j=1,\dots,J-1$. Thus,
along with \eqref{Ringer_1}, we obtain that
\begin{equation}\label{DCoe_12}
\begin{split}
|\eta(r)|_{\infty,h}\leq&\,\tfrac{h^2}{12} \,\,\max_{\ssy
I}|\psi_{yyyy}(r,\cdot)|\\
\leq&\,C\,h^2\,|\phi|_{2,\infty,{\ssy I}}.\\
\end{split}
\end{equation}
We use \eqref{DCoe_11} and \eqref{DCE_42} to have
\begin{equation}\label{DCoe_122}
\begin{split}
\|E(r)\|_{1,h}=&\,\|T_h(r)\eta(r)\|_{1,h}\\
\leq&\,C\,\|\eta\|_{0,h}\\
\leq&\,C\,|\eta|_{\infty,h}.\\
\end{split}
\end{equation}
If we combine \eqref{DCoe_122} and \eqref{DCoe_12},
\eqref{DCoe_10} easily follows.
\par
Now, assuming that $r\leq\tau$, we have
\begin{equation*}
\begin{split}
(\eta(r))_j-(\eta(\tau))_j=&\,\tfrac{h^2}{12}\,\left\{\,\left[\,
\psi_{yyyy}(r,\xi_j(r))-\psi_{yyyy}(\tau,\xi_j(r))\,\right]
+\left[\,\psi_{yyyy}(\tau,\xi_j(r))
-\psi_{yyyy}(\tau,\xi_j(\tau))\right]\,\right\}\\
=&\,\tfrac{h^2}{12}\,\left[\,(r-\tau)\,
\partial_r\partial_y^4\psi(\mu_j(r,\tau),\xi_j(r))
+(\xi_j(r)-\xi_j(\tau))\,\partial_y^5\psi(\tau,{\widetilde\xi}_j(r,\tau))
\,\right],\quad j=1,\dots,J-1,\\
\end{split}
\end{equation*}
with $\mu_j(r,\tau)\in(r,\tau)$ and
${\widetilde\xi}_j(r,\tau)\in(y_{j-1},y_{j+1})$ for
$j=1,\dots,J-1$. Thus, using \eqref{Ringer_1} and
\eqref{Cagaroo_1}, we obtain that
\begin{equation}\label{DCoe_170}
\begin{split}
|\eta(r)-\eta(\tau)|_{\infty,h}\leq&\,\tfrac{h^2}{12}
\,\,\left[\,|r-\tau|\,\max_{\ssy \Omega}|\psi_{yyyyr}|
+2\,h\,\max_{\ssy \Omega}|\psi_{yyyyy}|\,\right]\\
\leq&\,C\,h^2 \,\,\left(\,|r-\tau|\,|\phi|_{\infty,{\ssy I}}
+h\,|\phi|_{3,\infty,{\ssy I}}\,\right).\\
\end{split}
\end{equation}
Using \eqref{DCoe_11} we see that
\begin{equation}\label{DCoe_171}
\Lambda_h(r)(E(r)-E(\tau))=[\eta(r)-\eta(\tau)]+Z(r,\tau)\otimes
E(\tau)
\end{equation}
where $Z(r,\tau)\in X_h$ is defined by
\begin{equation}\label{DCoe_172}
(Z(r,\tau))_j:=-\left[\,-s^2(r)\,\tfrac{1+q\,\gamma(r,y_j)}{\alpha^2\,q}
+s^2(\tau)\,\tfrac{1+q\,\gamma(\tau,y_j)}{\alpha^2\,q}\,\right],\quad
j=1,\dots,J-1.
\end{equation}
Then, from \eqref{DCoe_172} we obtain
\begin{equation}\label{DCoe_173}
|Z(r,\tau)|_{\infty,h}\leq\,C\,|r-\tau|.
\end{equation}
We use \eqref{DCoe_171} and \eqref{DCE_42} to get
\begin{equation}\label{DCoe_174}
\begin{split}
\|E(r)-E(\tau)\|_{1,h}\leq&\,C\,\left(\,\|\eta(r)-\eta(\tau)\|_{0,h}
+\|Z(r,\tau)\otimes E(\tau)\|_{0,h}\,\right)\\
\leq&\,C\,\left(\,|\eta(r)-\eta(\tau)|_{\infty,h}
+|Z(r,\tau)|_{\infty,h}\,\|E(\tau)\|_{0,h}\,\right).\\
\end{split}
\end{equation}
Then, \eqref{DCoe_1000} follows easily from \eqref{DCoe_174},
\eqref{DCoe_173} and \eqref{DCoe_170}.
\end{proof}
%
%
%
%
%
%
%
%
%
\subsection{The Crank-Nicolson-type finite difference method \text{\rm
CNFD}}
%
%
\subsubsection{Formulation of the method \text{\rm CNFD}}\label{TheMethod}
%
%
%
%
For $m=0,\cdots,N$, the method CNFD constructs an approximation
$U^m$ of $u^m$ following the steps below:
\par
{\tt Step A1}: Define $U^0\in X_h$ by
\begin{equation}\label{FD1}
U^0:=P_hu_0.
\end{equation}
\par
{\tt Step A2}: For $n=1,\dots,N$, find $U^n\in X_h$ such that
\begin{equation}\label{FD21}
\tfrac{U^n-U^{n-1}}{k}-\delta(r^{n-\frac{1}{2}})
\,\,\omega\otimes\,\partial_h\left(\tfrac{U^n+U^{n-1}}{2}\right)
=\cunit\,\,\xi(r^{n-\frac{1}{2}})
\,\,T_h(r^{n-\frac{1}{2}})\left(\tfrac{U^n+U^{n-1}}{2}\right)
+\cunit\,\tfrac{\lambda}{q}\,\left(\tfrac{U^n+U^{n-1}}{2}\right),
\end{equation}
where $\delta(r):=\frac{{\dot s}(r)}{s(r)}$ for $r\in[0,R]$, and
$\xi(r):=\tfrac{\lambda\,s^2(r)}{\alpha^2\,q^2}$ for $r\in[0,R]$.
%
%

%
%
%
%
\subsubsection{Consistency of \text{\rm CNFD}}
%
%
%
%
%
%
\begin{lemma}\label{ld.4.5}
We assume that \eqref{DCE_40} or \eqref{DCE_60} hold. Also, for
$n=1,\dots,N$, we define $\eta^{n}\in X_h$ by
\begin{equation}\label{CC_equation_1}
\begin{split}
\tfrac{u^n-u^{n-1}}{k}-\delta(r^{n-\frac{1}{2}})\,\,\omega\otimes\partial_h\left(
\tfrac{u^n+u^{n-1}}{2}\right)=&\,\cunit\,\,\xi(r^{n-\frac{1}{2}})
\,T_h(r^{n-\frac{1}{2}})\left(\tfrac{u^n+u^{n-1}}{2}\right)\\
&\quad+\cunit\,\tfrac{\lambda}{q}\,\left(\tfrac{u^n+u^{n-1}}{2}\right)
+\eta^{n}.\\
\end{split}
\end{equation}
Then, there exists a constant $C>0$ depending on the data of the
problem \eqref{NewProblem2} which is such that
\begin{equation}\label{CCError_1}
\max_{1\leq{m}\leq{\ssy N}}|\eta^m|_{\infty,h}\leq\,C\,\left[\,
k^2\,{\mathcal B}_{1}(u)+h^2\,{\mathcal B}_{2}(u)\,\right]\\
\end{equation}
and
\begin{equation}\label{CCError_700}
\begin{split}
\max_{2\leq{n}\leq{\ssy
N}}|\eta^{n}-\eta^{n-1}|_{\infty,h}\leq&\,C\,\Big[\,k^3\,\Big(\,
|\partial_r^4u|_{\ssy \infty,\Omega}+|\partial_r^3u|_{\ssy
\infty,\Omega} +|\partial_r^2\partial_yu|_{\ssy
\infty,\Omega}+|\partial_r^3\partial_yu|_{\ssy
\infty,\Omega}\,\Big)\\
&\quad\quad+h^3\,\left(\,|\partial_y^4u|_{\ssy \infty,\Omega}+
\max_{[0,R]}|u|_{3,\infty,{\ssy I}}\,\right)\\
&\hskip-2.0truecm+k\,h^2\,\Big(\, |\partial_y^3u|_{\ssy
\infty,\Omega}+|\partial_r\partial_y^3u|_{\ssy \infty,\Omega}
+\max_{[0,R]}|u|_{2,\infty,{\ssy I}}
+\max_{[0,R]}|u_r|_{2,\infty,{\ssy I}}
\,\Big)\,\Big].\\
\end{split}
\end{equation}
where
\begin{equation*}
{\mathcal B}_{1}(u):=|u_{rrr}|_{\ssy
\infty,\Omega}+|u_{yrr}|_{\ssy \infty,\Omega}+|u_{rr}|_{\ssy
\infty,\Omega},\quad\quad\quad {\mathcal
B}_{2}(u):=\,\max_{[0,R]}|u|_{3,\infty,{\ssy I}}.
\end{equation*}
\end{lemma}
%
%
%
%
%
%
%
%
%
%
%
%
\begin{proof}
Let $n\in\{1,\dots,N\}$. Using Taylor's formula we conclude that
\begin{equation}\label{DNT_Spread_1A}
\begin{split}
\eta_j^{n}=&\,\tfrac{k^2}{24}\,u_{rrr}(\tau^n_j,y_j)
-\delta(r^{n-\frac{1}{2}})
\,y_j\,\left[\tfrac{k^2}{8}\,u_{yrr}({\widetilde\tau}^n_j,y_j)
+\tfrac{h^2}{6}\,u_{yyy}(r^n,z_j^n)
+\tfrac{h^2}{6}\,u_{yyy}(r^{n-1},{\widetilde z}^n_j)\right]\\
&-\cunit\,\,\xi(r^{n-\frac{1}{2}})
\,\left[T_h(r^{n-\frac{1}{2}})P_h\left(\tfrac{u(r^{n-1},\cdot)+u(r^n,\cdot)}{2}\right)
-P_hT(r^{n-\frac{1}{2}})\left(\tfrac{u(r^n,\cdot)+u(r^{n-1},\cdot)}{2}\right)\right]_j\\
&-\cunit\,\,\xi(r^{n-\frac{1}{2}})
\,\left[P_hT(r^{n-\frac{1}{2}})\left(\tfrac{u(r^n,\cdot)+u(r^{n-1},\cdot)}{2}
-u(r^{n-\frac{1}{2}},\cdot)\right) \right]_j\\
&+\cunit\,\tfrac{\lambda}{q}\,\tfrac{k^2}{8}\,u_{rr}(t^n_j,y_j),
\quad j=1,\dots,J-1,
\end{split}
\end{equation}
where $t^n_j$, $\tau^n_j$, ${\widetilde\tau}^n_j\in(r^{n-1},r^n)$
and ${\widetilde z}^n_j\in(y_{j-1},y_{j+1})$.
\par
Now, combine \eqref{DNT_Spread_1A}, \eqref{Max_H1_Dominus},
\eqref{DCoe_10}, \eqref{Ringer_1} and \eqref{Cagaroo_1} to obtain
\begin{equation}\label{DNT_Spread_1B}
\begin{split}
|\eta^{n}|_{\infty,h}\leq&\,C
\,\Big[k^2\,\big(\,|u_{rrr}|_{\infty,\ssy\Omega}+|u_{rr}|_{\infty,\ssy\Omega}
+|u_{yrr}|_{\infty,\ssy\Omega}\,\big)
+h^2\,|u_{yyy}|_{\infty,\ssy\Omega}\Big]
+C\,h^2\,\max_{[0,R]}|u|_{2,\infty,\ssy I}\\
&\quad+C\,\left|\tfrac{u(r^n,\cdot)+u(r^{n-1},\cdot)}{2}
-u(r^{n-\frac{1}{2}},\cdot)\right|_{\infty,{\ssy I}}\\
\leq&\,C\,\left[k^2\,\big(\,|u_{rrr}|_{\infty,\ssy\Omega}
+|u_{rr}|_{\infty,\ssy\Omega}
+|u_{yrr}|_{\infty,\ssy\Omega}\,\big)
+h^2\,\max_{[0,R]}|u|_{3,\infty,{\ssy I}}\right],
\end{split}
\end{equation}
and thus arrive at \eqref{CCError_1}.
\par
Since it holds that
\begin{equation*}
\begin{split}
\eta_j^{n}-\eta_j^{n-1}=&\,\tfrac{k^2}{24}\,\left[\,
u_{rrr}(\tau^n_j,y_j)-u_{rrr}(\tau^{n-1}_j,y_j)\right]
+\cunit\,\tfrac{\lambda}{q}\,\tfrac{k^2}{8}\,\left[\,
u_{rr}(t^n_j,y_j)-u_{rr}(t^{n-1}_j,y_j)\right]\\
&\,-\left[\delta(r^{n-\frac{1}{2}})-\delta(r^{n-\frac{3}{2}})
\right]\,y_j\,\left[\tfrac{k^2}{8}
\,u_{yrr}({\widetilde\tau}^n_j,y_j)+\tfrac{h^2}{6}\,u_{yyy}(r^n,z_j^n)
+\tfrac{h^2}{6}\,u_{yyy}(r^{n-1},{\widetilde z}^n_j)\right]\\
&-\delta(r^{n-\frac{3}{2}})\,y_j\,\Bigg\{\left[\tfrac{k^2}{8}
\,u_{yrr}({\widetilde\tau}^n_j,y_j)+\tfrac{h^2}{6}\,u_{yyy}(r^n,z_j^n)
+\tfrac{h^2}{6}\,u_{yyy}(r^{n-1},{\widetilde z}^n_j)\right]\\
&-\left[\tfrac{k^2}{8}
\,u_{yrr}({\widetilde\tau}^{n-1}_{j},y_j)+\tfrac{h^2}{6}
\,u_{yyy}(r^{n-1},z^{n-1}_{j})
+\tfrac{h^2}{6}\,u_{yyy}(r^{n-2},{\widetilde
z}_{n-1,j})\right]\Bigg\}\\
&-\cunit\,\left[\xi(r^{n-\frac{1}{2}})-\xi(r^{n-\frac{3}{2}})
\right]\,\left[T_h(r^{n-\frac{1}{2}})
P_h\left(\tfrac{u(r^n,\cdot)+u(r^{n-1},\cdot)}{2}\right)
-P_hT(r^{n-\frac{1}{2}})
\left(\tfrac{u(r^{n},\cdot)+u(r^{n-1},\cdot)}{2}\right)\right]_j\\
&-\cunit\,\xi(r^{n-\frac{3}{2}})
\,\Bigg\{\left[T_h(r^{n-\frac{1}{2}})P_h
\left(\tfrac{u(r^{n},\cdot)+u(r^{n-1},\cdot)}{2}\right)
-P_hT(r^{n-\frac{1}{2}})\left(\tfrac{u(r^{n},\cdot)
+u(r^{n-1},\cdot)}{2}\right)\right]_j\\
&\hskip1.0truecm-\left[T_h(r^{n-\frac{3}{2}})
P_h\left(\tfrac{u(r^{n-1},\cdot)+u(r^{n},\cdot)}{2}\right)
-P_hT(r^{n-\frac{3}{2}})
\left(\tfrac{u(r^{n-1},\cdot)+u(r^{n},\cdot)}{2}\right)\right]_j\Bigg\}\\
&-\cunit\,\xi(r^{n-\frac{3}{2}})
\,\Bigg\{\left[T_h(r^{n-\frac{3}{2}})
P_h\left(\tfrac{u(r^{n-1},\cdot)+u(r^{n},\cdot)}{2}\right)
-P_hT(r^{n-\frac{3}{2}})
\left(\tfrac{u(r^{n-1},\cdot)+u(r^{n},\cdot)}{2}\right)\right]_j\\
&\hskip1.0truecm-\left[T_h(r^{n-\frac{3}{2}})P_h
\left(\tfrac{u(r^{n-1},\cdot)+u(r^{n-2},\cdot)}{2}\right)
-P_hT(r^{n-\frac{3}{2}})\left(\tfrac{u(r^{n-1},\cdot)+u(r^{n-2},\cdot)}{2}
\right)\right]_j\Bigg\},\\
\end{split}
\end{equation*}
for $j=1,\dots,J-1$ and $n=2,\dots,N$, \eqref{CCError_700} follows
easily by the mean value theorem, \eqref{DCoe_10} and
\eqref{DCoe_1000}.
\end{proof}
%
%
%
%
%
%
%
%
%
%
%
%
%
%
%
%
%
\subsubsection{Convergence of \text{\rm CNFD}.}
%
%
%
%
%
%
%
%
\begin{thm}\label{L2Convergence}
Let $(U^m)_{m=0}^{\ssy J}\subset X_h$ be the finite difference
approximation to the solution of \eqref{NewProblem2} defined as in
Section~\ref{TheMethod}. Then, there exist constants $C_1>0$,
$C_2>0$ and $C_3\ge0$ independent of $k$ and $h$ such that: if
$C_3\,k\leq\, 3$, then
\begin{equation}\label{MaxL2Bound}
\max_{0\leq{n}\leq{\ssy
N}}\|U^n-u^n\|_{0,h}\leq\,C_1\,\max_{1\leq{n}\leq{\ssy
N}}\|\eta^n\|_{0,h}
\end{equation}
and
\begin{equation}\label{MaxMaxBound}
\max_{2\leq{n}\leq{\ssy
N}}\left|\omega\otimes\left(\tfrac{u^n+u^{n-1}}{2}-\tfrac{U^n+U^{n-1}}{2}
\right)\right|_{\infty,h}\leq\,C_2\,\Big[
k^{-1}\,\max_{2\leq{n}\leq{\ssy N}}\|\eta^{n}-\eta^{n-1}\|_{0,h}
+\max_{1\leq{n}\leq{\ssy N}}\|\eta^n\|_{0,h}\Big].
\end{equation}
\end{thm}
%
%
%
%
%
%
%
%
%
%
%
%
\begin{proof}
Let $e^n:=u^n-U^m$ for $n=0,\dots,N$, and observe that due to
\eqref{FD1} there holds that $e^0=0$. Next, subtract \eqref{FD21}
from \eqref{CC_equation_1} to obtain
\begin{equation}\label{EERREquation1}
\begin{split}
\tfrac{e^n-e^{n-1}}{k}-\delta(r^{n-\frac{1}{2}})
\,\omega\otimes\partial_h\left(
\tfrac{e^n+e^{n-1}}{2}\right)=&\,\cunit\,\xi(r^{n-\frac{1}{2}})
\,\,T_h(r^{n-\frac{1}{2}})\left(\tfrac{e^n+e^{n-1}}{2}\right)\\
&\quad+\cunit\,\tfrac{\lambda}{q}\,\left(\tfrac{e^n+e^{n-1}}{2}\right)
+\eta^{n},\quad n=1,\dots,N.
\end{split}
\end{equation}
%
%
%
%
%
%
%
%
%
%
%
%
%
%
%
%
%
%
%
%
%
%
%
Now, take the $(\cdot,\cdot)_{0,h}-$inner product of both sides of
\eqref{EERREquation1} with $e^n+e^{n-1}$, and then real parts to
get
\begin{equation}\label{EErrEstim2}
\begin{split}
\|e^n\|_{0,h}^2-\|e^{n-1}\|_{0,h}^2
=&\,\tfrac{k}{2}\,\delta(r^{n-\frac{1}{2}})\,\text{\rm
Re}\left(\omega\otimes\partial_h(e^n+e^{n-1}),e^n+e^{n-1}\right)_{0,h}\\
&\,-\tfrac{k}{2}\,\text{\rm Im}\left[\,\xi(r^{n-\frac{1}{2}})
\,\left(T_h(r^{n-\frac{1}{2}})(e^n+e^{n-1}),
e^n+e^{n-1}\right)_{0,h}\,\right]\\
&\,-\tfrac{k}{2}\,\text{\rm Im}\left[\,\tfrac{\lambda}{q}
\,\left(e^n+e^{n-1},e^n+e^{n-1}\right)_{0,h}\,\right]\\
&\,+k\,\text{\rm Re}\left(\eta^{n},
e^n+e^{n-1}\right)_{0,h},\quad n=1,\dots,N.\\
\end{split}
\end{equation}
Combining \eqref{EErrEstim2}, \eqref{ErrEstim3},
\eqref{Iden_Bound_1} and \eqref{DCE_42} we obtain
\begin{equation}\label{EErrEstim4}
\|e^n\|_{0,h}-\|e^{n-1}\|_{0,h}\leq\,k\,c_{\star}
\,\left(\,\,\|e^n\|_{0,h}+\|e^{n-1}\|_{0,h}\,\,\right)
+k\,\|\eta^n\|_{0,h},\quad n=1,\dots,N.
\end{equation}
Assuming that $k\,c_{\star}\leq\tfrac{1}{3}$, from
\eqref{EErrEstim4} we conclude that
\begin{equation*}
\|e^n\|_{0,h}\leq\,\tfrac{1+c_{\star}\,k}{1-c_{\star}\,k}
\,\|e^{n-1}\|_{0,h}+\tfrac{k}{1-c_{\star}\,k}\,\|\eta^{n}\|_{0,h},\quad
n=1,\dots,N,
\end{equation*}
which yields
\begin{equation}\label{nearxos1}
\|e^{n}\|_{0,h}\leq\,e^{4c_*k}\,\|e^{n-1}\|_{0,h} +
3\,k\,\|\eta^{n}\|_{0,h},\quad n=1,\dots,N.
\end{equation}
Applying a standard discrete Gr{\"o}nwall argument on
\eqref{nearxos1} and using the fact that $\|e^0\|_{0,h}=0$, we
obtain
\begin{equation}\label{nearxos2}
\begin{split}
\max_{0\leq{n}\leq{\ssy
N}}\|e^{n}\|_{0,h}\leq&\,C\,\big(\,\|e^0\|_{0,h}+\max_{1\leq{n}\leq{\ssy
N}}\|\eta^n\|_{0,h}\,\big)\\
\leq&\,C\,\max_{1\leq{n}\leq{\ssy
N}}\|\eta^n\|_{0,h},\\
\end{split}
\end{equation}
which establishes the estimate \eqref{MaxL2Bound}.
\par
Let $\nu^m:=e^m-e^{m-1}$ for $m=1,\dots,N$. Then,
\eqref{EERREquation1} yields
\begin{equation}\label{BigGoal1}
\begin{split}
\tfrac{\nu^{n}-\nu^{n-1}}{k}-\delta(r^{n-\frac{1}{2}})
\,\omega\otimes\partial_h \left(\tfrac{\nu^n+\nu^{n-1}}{2}\right)
&=\,\left[\delta(r^{n-\frac{1}{2}})-\delta(r^{n-\frac{3}{2}})
\right]\,\omega\otimes\partial_h
\left(\tfrac{e^{n-1}+e^{n-2}}{2}\right)\\
&+\,\cunit\,\xi(r^{n-\frac{3}{2}}) \,T_h(r^{n-\frac{1}{2}})
\left(\tfrac{\nu^{n}+\nu^{n-1}}{2}\right)\\
&+\,\cunit\,\left[\,\xi(r^{n-\frac{1}{2}})
-\xi(r^{n-\frac{3}{2}})\,\right]
\,T_h(r^{n-\frac{1}{2}})\left(\tfrac{e^n+e^{n-1}}{2}\right)\\
&+\,\cunit\,\xi(r^{n-\frac{3}{2}})
\,\left[\,T_h(r^{n-\frac{1}{2}})\left(\tfrac{e^n+e^{n-1}}{2}\right)
-T_h(r^{n-\frac{3}{2}})\left(\tfrac{e^n+e^{n-1}}{2}\right)\,\right]\\
&+\,\cunit\,\tfrac{\lambda}{q}\,\left(\tfrac{\nu^n+\nu^{n-1}}{2}\right)\\
&+(\eta^{n}-\eta^{n-1}),\quad n=2,\dots,N.\\
\end{split}
\end{equation}
Now, take the $(\cdot,\cdot)_{0,h}-$inner product of both sides of
\eqref{BigGoal1} with $\nu^n+\nu^{n-1}$ and then real parts to get
\begin{equation}\label{BigGoal2}
\begin{split}
\|\nu^{n}\|_{0,h}^2-\|\nu^{n-1}\|_{0,h}^2\leq&\,C\,k\,\Big[\,
\|\nu^n+\nu^{n-1}\|_{0,h}+k\,\|e^n+e^{n-1}\|_{0,h}
+\|\eta^{n}-\eta^{n-1}\|_{0,h}\\
&\hskip1.0truecm+\left\|T_h(r^{n-\frac{1}{2}})(e^n+e^{n-1})
-T_h(r^{n-\frac{3}{2}})(e^n+e^{n-1})\right\|_{0,h}\\
&\hskip1.0truecm+k\,\left\|\omega\otimes\partial_h
\left(\tfrac{e^{n-1}+e^{n-2}}{2}\right)\right\|_{0,h}\,\Big]
\,\|\nu^n+\nu^{n-1}\|_{0,h},\quad n=2,\dots,N,\\
\end{split}
\end{equation}
after using \eqref{DCE_42}.
Now, \eqref{EERREquation1} along with \eqref{DCE_42}, yield
\begin{equation}\label{BigGoal3}
\left\|\omega\otimes\partial_h
\left(\tfrac{e^{n-1}+e^{n-2}}{2}\right)\right\|_{0,h}\leq\,C\,\Big[\,
\|\eta^{n-1}\|_{0,h}+\|e^{n-1}\|_{0,h}+\|e^{n-2}\|_{0,h}
+k^{-1}\,\,\|\nu^{n-1}\|_{0,h}\Big],\quad n=2,\dots,N.
\end{equation}
Thus, \eqref{BigGoal3} and \eqref{BigGoal2} yield
\begin{equation}\label{TomTom1}
\begin{split}
\|\nu^{n}\|_{0,h}-\|\nu^{n-1}\|_{0,h}\leq&\,C\,k\,\Big[\,
\|\nu^n\|_{0,h}+\|\nu^{n-1}\|_{0,h}+k\,\left(\|e^n\|_{0,h}+\|e^{n-1}\|_{0,h}
+\|e^{n-2}\|_{0,h}\right)\\
&\hskip1.0truecm +k\,\|\eta^{n-1}\|_{0,h}
+\|\eta^{n}-\eta^{n-1}\|_{0,h}\\
&\hskip1.0truecm+\left\|T_h(r^{n-\frac{1}{2}})(e^n+e^{n-1})
-T_h(r^{n-\frac{3}{2}})(e^n+e^{n-1})\right\|_{0,h}
\,\Big]\\
\end{split}
\end{equation}
for $n=2,\dots,N$.
Now, we observe that, for $v\in X_h$, we have
\begin{equation}\label{TomTom2}
\begin{split}
\Lambda_h(r^{n-\frac{1}{2}})\left(T_h(r^{n-\frac{1}{2}})v
-T_h(r^{n-\frac{3}{2}})v\right)=&\,
\Lambda_h(r^{n-\frac{3}{2}})T_h(r^{n-\frac{3}{2}})v
-\Lambda_h(r^{n-\frac{1}{2}})T_h(r^{n-\frac{3}{2}})v\\
=&\,V^n\otimes T_h(r^{n-\frac{3}{2}})v,\quad n=2,\dots,N,\\
\end{split}
\end{equation}
where $V^n\in X_h$ is given by
\begin{equation}\label{TomTom3}
V^n_j:=s^2(r^{n-\frac{3}{2}})\,
\tfrac{1+q\,\gamma(r^{n-\frac{3}{2}},y_j)}{\alpha^2\,q}
-s^2(r^{n-\frac{1}{2}})
\,\tfrac{1+q\,\gamma(r^{n-\frac{1}{2}},y_j)}{\alpha^2\,q},\quad
j=1,\dots,J-1.
\end{equation}
Using \eqref{TomTom2}, \eqref{DCE_42} and \eqref{TomTom3} we have
\begin{equation}\label{BigGoal4}
\begin{split}
\|T_h(r^{n-1})v-T_h(r^{n-2})v\|_{0,h}
\leq&\,C\,|V^n|_{\infty,h}\,\|v\|_{0,h}\\
\leq&\,C\,k\,\|v\|_{0,h},\quad\forall\,v\in X_h,
\quad n=2,\dots,N.\\
\end{split}
\end{equation}
Combining \eqref{TomTom1} and \eqref{BigGoal4} we obtain
\begin{equation}\label{BigGoal5}
\begin{split}
\|\nu^{n}\|_{0,h}-\|\nu^{n-1}\|_{0,h}\leq&\,C_{\star}\,k\,\Big[\,
\|\nu^n\|_{0,h}+\|\nu^{n-1}\|_{0,h}
+k\,\left(\|e^n\|_{0,h}+\|e^{n-1}\|_{0,h}
+\|e^{n-2}\|_{0,h}\right)\\
&\hskip1.0truecm+k\,\|\eta^{n-1}\|_{0,h}
+\|\eta^{n}-\eta^{n-1}\|_{0,h}\,\Big],
\quad n=2,\dots,N.\\
\end{split}
\end{equation}
Assuming that $k$ is enough small (i.e.
$3\,k\,\max\{c_{\star},C_{\star}\}\,\leq1$) and applying a
discrete Gr{\"o}nwall argument on \eqref{BigGoal5}, we conclude
that
\begin{equation*}
\begin{split}
\max_{1\leq{n}\leq{\ssy N}}\|\nu^n\|_{0,h}\leq\,C\,\Big[&
\max_{2\leq{n}\leq{\ssy N}}\|\eta^{n}-\eta^{n-1}\|_{0,h}
+k\,\max_{1\leq{n}\leq{\ssy N-1}}\|\eta^n\|_{0,h}\\
&+k\,\max_{0\leq{n}\leq{\ssy N}}\|e^n\|_{0,h}
+\|\nu^1\|_{0,h}\Big].\\
\end{split}
\end{equation*}
which, along with \eqref{MaxL2Bound} and \eqref{nearxos1}, yields
\begin{equation}\label{BigGoal6}
\begin{split}
\max_{1\leq{n}\leq{\ssy N}}\|\nu^n\|_{0,h}\leq&\,C\,\Big[
\max_{2\leq{n}\leq{\ssy N}}\|\eta^{n}-\eta^{n-1}\|_{0,h}
+k\,\max_{1\leq{n}\leq{\ssy N}}\|\eta^n\|_{0,h}
+\|e^1\|_{0,h}\Big]\\
\leq&\,C\,\Big[ \max_{2\leq{n}\leq{\ssy
N}}\|\eta^{n}-\eta^{n-1}\|_{0,h} +k\,\max_{1\leq{n}\leq{\ssy
N}}\|\eta^n\|_{0,h}\Big].\\
\end{split}
\end{equation}
\par
Now, from \eqref{EERREquation1} follows that
\begin{equation}\label{BigGoal10}
\|\omega\otimes\partial_h
(e^n+e^{n-1})\|_{0,h}\leq\,C\,\Big[\,k^{-1}\,\|\nu^{n}\|_{0,h} +
\|e^{n}\|_{0,h} +\|e^{n-1}\|_{0,h}
+\|\eta^{n}\|_{0,h}\,\Big],\quad n=1,\dots,N.
\end{equation}
Then, \eqref{BigGoal10} and \eqref{BigGoal6} yield that
\begin{equation}\label{BigGoal11}
\max_{1\leq{n}\leq{\ssy N}}\|\omega\otimes\partial_h
(e^n+e^{n-1})\|_{0,h}\leq\,C\,\Big[
k^{-1}\,\max_{2\leq{n}\leq{\ssy N}}\|\eta^{n}-\eta^{n-1}\|_{0,h}
+\max_{1\leq{n}\leq{\ssy N}}\|\eta^n\|_{0,h}\Big].
\end{equation}
Finally, \eqref{MaxMaxBound} follows from \eqref{BigGoal11} and
\eqref{DSobolev_Bound}, in view of the identity
$\partial_h(\omega\otimes v)=\omega\otimes\partial_hv+I_hv$ for
$v\in X_h$.
\end{proof}

In view of the results of {\rm Lemma}~\ref{ld.4.5},
\eqref{MaxL2Bound} implies that
\begin{equation*}
\displaystyle{\max_{0\leq n\leq{\ssy N}}}\|U^n-u^n\|_{\ssy
0,h}=\mathcal{O}(k^2+h^2).
\end{equation*}
Also, the estimates \eqref{MaxMaxBound}, \eqref{CCError_1}
and \eqref{CCError_700} yield
\begin{equation*}
\displaystyle{\max_{0\leq n\leq{\ssy N}}}
|\omega\otimes e^{n-\frac{1}{2}}|_{\ssy
\infty,h}=\mathcal{O}(k^2+h^2+h^3\,k^{-1})
\end{equation*}
which may be viewed as a discrete weighted maximum norm estimate
on $e^{n-\frac{1}{2}}$, which is of optimal order when $h=\mathcal{O}(k)$;
however does not yield an optimal order maximum norm estimate, due
to vanishing of the coefficient of $u_y$ in the p.d.e. in \eqref{NewProblem2}
at $y=0$.
%

%
%
%
%
%
\section{Numerical implementation}\label{sec4}
\subsection{The numerical scheme}
Using the definitions $\delta(r):=\tfrac{{\dot s}(r)}{s(r)}$,
$\zeta(r,y):=\tfrac{(1+q\,\gamma(r,y))\,s^2(r)}{\alpha^2\,q}$,
$\xi(r):=\tfrac{\lambda\,s^2(r)}{\alpha^2\,q^2}$, we note that
the p.d.e. in \eqref{NewProblem2} may written as
\begin{equation*}\label{n1}
-\zeta(r,y)G(r,y) -\partial_y^2G(r,y)=\cunit\xi(r)\,u,
\end{equation*}
where we have put
$G(r,y):=u_r(r,y)-\cunit\,\tfrac{\lambda}{q}\,u(r,y)-\delta(r)y\,u_y(r,y)$.
Motivated by this we rewrite the CNFD scheme
\eqref{FD1}-\eqref{FD21} in the following equivalent form using
the notation introduced in section~\ref{sec3} and putting
$U_j^{n-\frac{1}{2}}:=\tfrac{U_j^{n-1}+U_j^n}{2}$. We seek
$U_j^n$, $0\leq j\leq J$, $0\leq n\leq N$, approximating $u_j^n$
and given by the equations:\\
For $n=0$:
\begin{equation}\label{d5.1}
U_j^0=u_0(y_j)\quad\mbox{for}\quad 1\leq j\leq J-1,
\quad U_0^0=U_{\ssy J}^0=0.
\end{equation}
For $n=1,\dots,N$:
\begin{equation}\label{d5.2}
-\zeta(r^{n-\frac{1}{2}},y_j)\,G_j^n-\widetilde{\Delta}_h
G_j^n=\cunit\,\xi(r^{n-\frac{1}{2}})\,U_j^{n-\frac{1}{2}}
\quad\mbox{for}\quad 1\leq j\leq J-1,
\end{equation}
where
$$\widetilde{\Delta}_h G_j^n:=
  \begin{cases}
    \frac{G_2^n-2G_1^n}{h^2} & j=1, \\
    \frac{G_{j+1}^n-2G_j^n+G_{j-1}^n}{h^2} & j=2,\cdots,J-2,\\
    \frac{-2G_{J-1}^n+G_{J-2}^n}{h^2}  & j=J-1,
  \end{cases}$$
with
\begin{equation*}
G_j^n:=\tfrac{{U_j^n}-U_j^{n-1}}{k}-\cunit\,\tfrac{\lambda}{q}\,U_j^{n-\frac{1}{2}}
-\delta(r^{n-\frac{1}{2}})\,y_j\,\tfrac{U_{j+1}^{n-\frac{1}{2}}-U_{j-1}^{n-\frac{1}{2}}}{2h}
\quad\mbox{for}\quad j=1,\dots,J-1.
\end{equation*}
The scheme requires  solving a pentadiagonal linear system of algebraic equations
at each time step.
%
\subsection{Numerical experiments}
%
We implemented the finite difference scheme CNFD in the form
\eqref{d5.1}-\eqref{d5.2} in a double precision FORTRAN 77 code
using it in various numerical examples to test its accuracy and
stability. We used the function $u(r,y)=\exp(2r)(y-1)\sin(2\pi
y)$ as exact solution of \eqref{NewProblem2} (with an appropriate
nonhomogeneous term in the right-hand side of the p.d.e.),
putting $\gamma(r,y)=1+y$, $\alpha=2$, $p=q+\frac{1}{2}$, and selecting
$q=(0.252252311,-1.35135138\,e-2)$, \cite{Co}.
We experimented with several bottom profiles. For example, in the case of the
downsloping bottom given by $s(r)=\exp(r)$, $0\leq r\leq 1$, we
obtained at $r=1$ the errors in the discrete $L^2$ and
$L^{\infty}$ norms shown in Table~\ref{pinakas1} together with the associated
experimental convergence rates. (We took $h=k=\frac{1}{J}$ for the values
of $J$ shown). The convergence rates of the table are practically
equal to 2 and consistent with the predictions of the theory. We
also tried bottom profiles $s(r)$ given for $0\leq r\leq 1$ by $r+2$,
$-r+2$, $-\exp(-r)$, $\cos(2\pi r)+2$, i.e. upsloping and
oscillating profiles as well, and found experimentally that the
scheme was stable for $k=h$ and that the $L^2$- and
$L^{\infty}$-convergence rates were practically equal to $2$
again. Hence it seems that although CNFD is designed to
approximate the ibvp \eqref{NewProblem2} in the downsloping
bottom case, it is resilient enough in upsloping and
non-monotonic bottom problems as well.
\begin{table}[b]
\begin{center}
\begin{tabular}{|c|c|c|c|c|}\hline
   $J$& $L^2$-error & $L^2$-rate  & $L^{\infty}$-error & $L^{\infty}$-rate\\ \hline
  40 &  0.2510\,e-1 & & 0.2493\,e-1 & \\ \hline
  80 & 0.6424\,e-2 &1.966 & 0.6365\,e-2 & 1.969\\ \hline
 160 &  0.1627\,e-2&1.981 & 0.1609\,e-2 & 1.983\\ \hline
 320 &  0.4097\,e-3&1.990 & 0.4048\,e-3 & 1.991\\ \hline
 640 &  0.1028\,e-3&1.995 & 0.1015\,e-3 & 1.995\\ \hline
1280 &  0.2574\,e-4&1.997 &0.2542\,e-4 &1.998\\ \hline
\end{tabular}
\end{center}
\vspace{0.4cm} \caption{Discrete $L^2$- and $L^{\infty}$-errors
and rates at $r=1$, $s(r)=\exp(r)$, $k=h=\frac{1}{J}$.}\label{pinakas1}
\end{table}
\par
Although the ibvp \eqref{NewProblem2} is $L^2$-conservative, in
the sense that its solution preserves the quantity
\begin{equation*}
\sqrt{s(r)}\|u(r,\cdot)\|_{L^2(I)}
\end{equation*}
for real $\gamma$ and $q$, CNFD does not share a corresponding
discrete property. For example, when we integrated numerically
\eqref{NewProblem2} with $\alpha=10$, $q=\frac{1}{4}$, $p=q+\frac{1}{2}$,
$\gamma(r,y)=1+y$, $u_0(y)=y^2(y-1)$, $0\leq y\leq 1$,
we found that in the cases $s(r)=\exp(r)$ and $s(r)=r+2$, $0\leq r\leq 1$, the
quantity $\sqrt{s(r^n)}\|U^n\|_{\ssy 0,h}$ was preserved to $4$
significant digits.
\par
We also performed a simulation of a realistic underwater
acoustics problem using the CNFD scheme. We integrated the ibvp
\eqref{NewProblem2} using again
$q=(0.252252311,-1.35135138\,e-2)$, $p=q+\frac{1}{2}$, and considered a
straight downsloping bottom given by
$s(r)=200\,\left(1+\frac{r}{4000}\right)$ (distances in meters)
and making an angle of $2.86^\circ$ with respect to the
horizontal surface. We used as initial condition at $r=0$ the
normal mode starter given by formula $(45)$ of \cite{DSZ2009} with
$M=6$, simulating the initial field produced by a time-harmonic
point source of frequency ${\rm f}=25$ Hz located at a depth of
$z=100$ m. (We shall frequently refer in the sequel to
quantities expressed in the $r$, $z$ variables. Of course the
scheme is implemented in the $r$, $y$ variables and the results
are transformed from or into the $r$, $z$ domain as required.) We
assume that the medium is homogeneous and lossless with a sound
speed $c=c_0=1500\,{\rm m/sec}$, so that $\beta=0$. We integrated
the problem up to $R=3300\,$m using $k=0.83475\,$m, and $4000$
mesh intervals of equal length in $y$. As is customary in
underwater acoustics we present the numerical results in terms of
a one-dimensional transmission loss (TL) plot in the $r$, $z$
variables. (The TL function was computed by the formula
TL$=-20\,\log_{10}\,\left(\frac{|v(r,z_{{\rm rec}})|}{\sqrt{r}}\right)$ where
$z_{\rm rec}$ is a receiver depth.)
The graph of Figure~\ref{f2} shows as an example, the TL curves at
$z_{{\rm rec}}=30\,$m obtained by integrating
CNFD (solid line) and by the finite difference scheme (dotted
line, `DSZ scheme') proposed in \cite{DSZ2009} as a
discretization of a problem of the form \eqref{NewProblem2} but
with the bottom boundary condition (cf. Remark \ref{rd3.3})
\begin{equation*}
u_{yy}(r,1)={\rm i}\,\tfrac{2}{\alpha}\,s(r)\,{\dot s}(r)\,u_y(r,1),
\quad 0\leq r\leq R,
\end{equation*}
replacing $u_y(r,1)=0$. The results of the two schemes are in very
good agreement. (In \cite{DSZ2009}, the results of the DSZ scheme
were compared to those of a standard `staircase' wide-angle PE
code and were found to be in very good agreement.)
\begin{figure}[ht]
\centering
\includegraphics[width=0.85\textwidth,height=0.30\textheight]{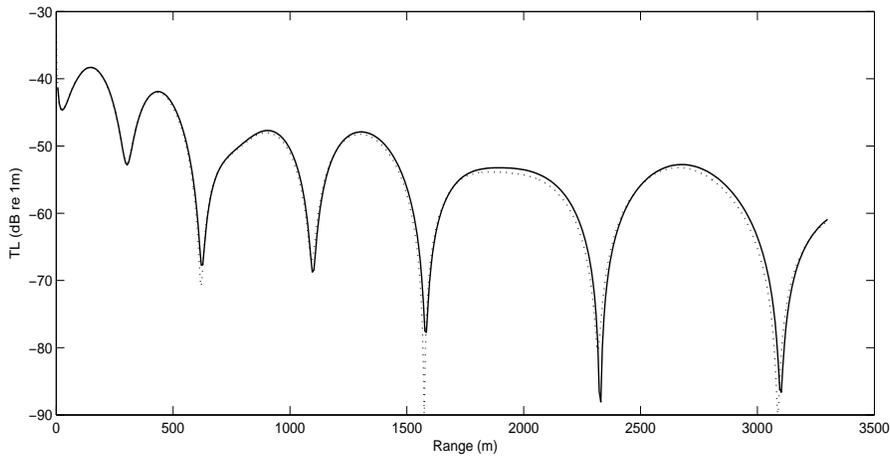}
\caption{TL at $z_{{\rm rec}}=30$ m, downsloping bottom. CNFD
(solid line), DSZ (dotted line).}\label{f2}
\end{figure}

\end{document}